\documentclass[11pt]{amsart}
\usepackage{url}
\usepackage[cmtip, all]{xy}

\usepackage[export]{adjustbox}
\usepackage{graphicx}

\usepackage{xcolor}

\theoremstyle{plain}
\newtheorem{thm}{Theorem}[section]
\newtheorem{theorem}[thm]{Theorem}

\newtheorem{lemma}[thm]{Lemma}
\newtheorem{corollary}[thm]{Corollary}
\newtheorem{proposition}[thm]{Proposition}
\newtheorem{lemma-definition}[thm]{Lemma-Definition}
\theoremstyle{definition}
\newtheorem{remark}[thm]{Remark}

\newtheorem{definition}[thm]{Definition}
\newtheorem{claim}[thm]{Claim}

\newtheorem{example}[thm]{Example}

\newtheorem{conjecture}[thm]{Conjecture}

\newtheorem{question}[thm]{Question}

\numberwithin{equation}{section}

\newcommand{\ml}[2]{\begin{multline}\label{#1}#2 \end{multline}}
\newcommand{\ga}[2]{\begin{gather}\label{#1}#2 \end{gather}}
\newcommand{\gr}{{\rm gr}}

\newcommand{\Pic}{{\rm Pic}}

\newcommand{\Spec}{{\rm Spec \,}}

\newcommand{\sC}{{\mathcal C}}
\newcommand{\sD}{{\mathcal D}}
\newcommand{\sE}{{\mathcal E}}
\newcommand{\sF}{{\mathcal F}}

\newcommand{\sL}{{\mathcal L}}

\newcommand{\sO}{{\mathcal O}}

\newcommand{\sS}{{\mathcal S}}

\newcommand{\sU}{{\mathcal U}}
\newcommand{\sV}{{\mathcal V}}

\newcommand{\A}{{\mathbb A}}

\newcommand{\C}{{\mathbb C}}

\newcommand{\F}{{\mathbb F}}
\newcommand{\G}{{\mathbb G}}

\newcommand{\N}{{\mathbb N}}

\newcommand{\Q}{{\mathbb Q}}

\newcommand{\Z}{{\mathbb Z}}

\newcommand{\id}{{\rm id\hspace{.1ex}}}

\DeclareMathOperator{\Ab}{\mathbb{A}}

\DeclareMathOperator{\Sym}{\mathsf{Sym}}

\DeclareMathOperator{\Eend}{\underline{\mathsf{End}}}

\DeclareMathOperator{\Oo}{\mathcal{O}}

\DeclareMathOperator{\Higgs}{\mathsf{Higgs}}
\DeclareMathOperator{\MIC}{\mathsf{MIC}}

\DeclareMathOperator{\MdR}{\mathcal{M}_{dR}}
\DeclareMathOperator{\MdRrig}{\mathcal{M}_{dR}^{rig}}

\DeclareMathOperator{\MDol}{\mathcal{M}_{Dol}}
\DeclareMathOperator{\MDolrig}{\mathcal{M}_{Dol}^{rig}}

\DeclareMathOperator{\MB}{\mathcal{M}_B}
\DeclareMathOperator{\MBrig}{\mathcal{M}_{B}^{rig}}

\DeclareMathOperator{\MHod}{\mathcal{M}_{Hod}}
\DeclareMathOperator{\MHodrig}{\mathcal{M}_{Hod}^{rig}}

\DeclareMathOperator{\Cb}{\mathbb{C}}
\DeclareMathOperator{\Zb}{\mathbb{Z}}
\DeclareMathOperator{\Qb}{\mathbb{Q}}
\DeclareMathOperator{\Pb}{\mathbb{P}}
\DeclareMathOperator{\Fb}{\mathbb{F}}

\DeclareMathOperator{\Ee}{\mathcal{E}}

\DeclareMathOperator{\Isoc}{\mathsf{Isoc}}

\title [Rigid connections and $F$-isocrystals]{Rigid connections and $F$-isocrystals}
\author{H\'el\`ene Esnault \and Michael Groechenig} 
\address{Freie Universit\"at Berlin, Arnimallee 3, 14195, Berlin,  Germany}
\email{esnault@math.fu-berlin.de}
\address{University of Toronto, 40 St. George Street, Toronto Ontario M5S 2E4, Canada}
\email{michael.groechenig@utoronto.ca}
\thanks{The first  author is supported by  the Einstein program and the ERC
Advanced
Grant 226257, the second author was supported by a Marie Sk\l odowska-Curie fellowship: This project has received funding from the European Union's Horizon 2020 research and innovation programme under the Marie Sk\l odowska-Curie Grant Agreement No. 701679.  \\ \includegraphics[height = 1cm,right]{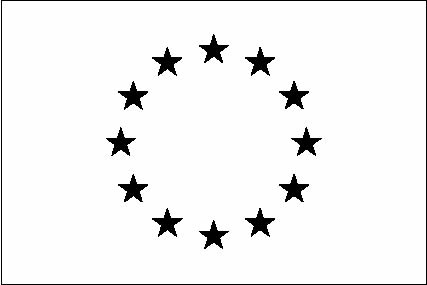}
}

\date{}
\begin{document}
\begin{abstract}
{An irreducible integrable connection $(E,\nabla)$ on a smooth projective complex variety $X$ is called rigid if it gives rise to an isolated point of the corresponding moduli space $\MdR(X)$.} According to Simpson's motivicity conjecture, {irreducible} rigid flat connections are of geometric origin, that is, arise as subquotients of a Gau\ss-Manin connection of a family of smooth projective varieties defined on an open dense subvariety of $X$. In this article we study mod $p$ reductions of {irreducible} rigid connections and establish results which confirm Simpson's prediction. In particular, for large $p$, we prove that $p$-curvatures of mod $p$ reductions of {irreducible} rigid flat connections are nilpotent, and building on this result, we construct an $F$-isocrystalline realization for {irreducible} rigid flat connections. More precisely, we prove that there exist smooth models $X_R$ and $(E_R,\nabla_R)$ of $X$ and $(E,\nabla)$, over a finite type ring $R$, such that for every Witt ring $W(k)$ of a finite field $k$ and every homomorphism $R \to W(k)$, the $p$-adic completion of the base change $(\widehat{E}_{W(k)},\widehat{\nabla}_{W(k)})$ on $\widehat{X}_{W(k)}$ represents an $F$-isocrystal. Subsequently we show that {irreducible} rigid flat connections with vanishing $p$-curvatures are unitary. This allows us to prove new cases of the Grothendieck--Katz $p$-curvature conjecture. We also prove the existence of a complete companion correspondence
for $F$-isocrystals stemming from {irreducible} cohomologically rigid connections.
\end{abstract}
\maketitle

\tableofcontents

\section{Introduction}\label{intro}

This article is concerned with a conjecture of Simpson about irreducible flat connections $(E,\nabla)$ on a smooth projective variety $X/\Cb$. In \cite[Theorem 4.7]{Sim94}, Simpson constructed a quasi-projective moduli space $\MdR(X,r)$ of irreducible flat connections $(E,\nabla)$ of rank $r$ on $X$. It follows that for a {rank one flat connection  $L=(\sL, \nabla_{\sL})$} one has a moduli space $\MdR(X,L,r)$ {of}  irreducible flat connections $(E,\nabla)$ of rank $r$ on $X$, together with an isomorphism  {${\rm det} (E,\nabla)\simeq L$.} 

\begin{definition} \label{def:rigid}
An irreducible rank $r$ flat connection $(E,\nabla)$ with determinant line bundle $L$ is \emph{rigid} if the corresponding point of the moduli space $[(E,\nabla)] \in \MdR(X,L,r)$ is isolated.  
\end{definition}

\begin{remark}\label{rmk:standing}
Henceforth the term \emph{rigid connection} refers to a {stable} flat connection which satisfies the assumptions of Definition \ref{def:rigid}. {For the field of complex numbers, stability is equivalent to irreducibility.} Furthermore, we shall assume that $L$ {is \emph{torsion}} on $X$.
\end{remark}

Amongst the irreducible flat connections on $X$ there are specimens which stand out with particularly interesting properties: flat connections of \emph{geometric origin}. The latter are precisely subquotients of Gau\ss-Manin connections, {that is, with underlying local system a summand of} $R^if_*\Cb$, where $f\colon Y \to U$ is a smooth projective morphism, with a dense open subvariety $U \subset X$ as target. According to a conjecture by Simpson (see \cite[p. 9]{Sim92}) rigid flat connections are expected to possess this property.

\begin{conjecture}[Simpson's Motivicity Conjecture] \label{conj:mot}
A rigid flat connection $(E,\nabla)$ on $X$ with torsion determinant line bundle is of \emph{geometric origin}.
\end{conjecture}
 
For now this remains out of reach, yet there is a lot of supporting evidence:
\begin{enumerate}
\item Non-abelian Hodge theory implies that rigid flat connections give rise to { \emph{complex variations of Hodge structure} }on $X$. This was observed by Simpson in \cite[Section 4]{Sim92}. We refer the reader to Section \ref{sec:proofthmnilp}, where we give a short summary of Simpson's argument and also explain the connection with the present work. 
\item Non-abelian Hodge theory methods were used by Corlette--Simpson \cite{CS08} and more recently Simpson--Langer \cite{LS16}, to establish the Motivicity Conjecture for rigid flat connections with topological monodromy defined over the ring of algebraic integers $\bar{\Zb}$ of rank $2$, respectively $3$.
\item For the case of rigid flat connections on $\Pb^1$ minus finitely many points, there is a complete classification due to Katz \cite{Kat96} which implies Simpson's conjecture in this case. 
\end{enumerate}

It was shown by Katz (see \cite[Theorem 10.0]{Kat70} and \cite[3.1]{Kat72}) that Gau\ss-Manin connections in characteristic $p$ have nilpotent $p$-curvatures. This implies that their subquotients,   {which are by definition}  flat connections of geometric origin, also have nilpotent $p$-curvatures. We will recall the basics on $p$-curvature in Subsection \ref{sub:p-curvature}, and in particular explain Katz's theorem in the discussion above Theorem \ref{thm:katz}. 

Simpson's conjecture therefore predicts that mod $p$ reductions of rigid flat connections have nilpotent $p$-curvatures, at least for $p$ sufficiently large. Our first main result confirms this expectation.
 
\begin{thm}[Nilpotent $p$-curvature] \label{thm:nilp}
Let $X$ be a smooth connected projective complex variety and $(E,\nabla)$ be a rigid flat connection with torsion determinant $L$. Then there is a scheme $S$ of finite type over $\Z$ over which $(X, (E,\nabla))$ has a model $(X_S, (E_S,\nabla_S))$ such that for all closed points $s\in S$,  $(E_s,\nabla_s)$ has nilpotent $p$-curvature. 
\end{thm}

After replacing $S$ by an open dense subscheme, we may assume that $X_S \to S$ is smooth. For every Witt ring $W(k)$ of a finite field $k$, and every morphism $\Spec W(k) \to S$ one obtains a formal flat connection $(\widehat{E}_{W(k)},\widehat{\nabla}_{W(k)})$ on the $p$-adic completion 
$$\widehat{X}_{W(k)} = \left((X \times_S \Spec W(k)\right)^{\widehat{\;}}.$$ 
Furthermore, by the theorem above, the $p$-curvature of the restriction of this formal connection to $X_k = X \times_S \Spec k$ is nilpotent. A formal connection with this property gives rise to a \emph{crystal} on {$X_k$}. We refer the reader to Subsection \ref{sub:criss} for more details and references.
 
\begin{corollary}\label{cor:isoc}
Let $(E,\nabla)$ and $(E_S,\nabla_S)$ be as in Theorem \ref{thm:nilp}, $k$ and let $\Spec W(k) \to S$ be a morphism which factors through the smooth locus of $S$, and where $k$ is a finite field. Then the pullback to the formal scheme (obtained by $p$-adic completion) $$(\hat E_{W(k)},  \hat \nabla_{W(k)})$$ defines a crystal on {$X_k/W(k)$. }
\end{corollary}

The result above is a direct corollary of Theorem \ref{thm:nilp} (nilpotency of $p$-curvature). Our second main result generalizes a result of Crew for rigid flat connections on $\Pb^1$ minus finitely many points \cite[Theorem 3]{Cre17}. We show that the crystals associated to rigid flat connections in Corollary \ref{cor:isoc} do in fact give rise to $F$-isocrystals. We recall that for $\Spec k \to S$ as above, the category of isocrystals $\Isoc(X_k)$ on $X_k$ is defined as the $\Qb$-linearization of the category of crystals on {$X_k/W(k)$.} The Frobenius $F\colon X_k \to X_k$ allows one to define an endofunctor $F^*\colon \Isoc(X_k) \to \Isoc(X_k)$ (see \cite[Corollaire 1.2.4]{Ber74} for details). We say that a crystal $\Ee$ has a \emph{Frobenius structure}, if there exists a positive integer $f$ such that $(F^*)^f(\Ee) \simeq \Ee$.

\begin{thm}[$F$-isocrystals] \label{thm:Fstr}
 Let $X$ and $(E,\nabla)$ be as in Corollary \ref{cor:isoc}. Then there is a scheme $S$ of finite type over $\Z$ over which $(X, (E,\nabla))$ has a model $(X_S, (E_S,\nabla_S))$ such that for all  {$W(k)$-points of $S$},  the isocrystal $(\widehat{E}_{W(k)}, \widehat{\nabla}_{W(k)})\otimes \Qb$ has a Frobenius structure {after base change to a finite field extension $k'/k$}. 
\end{thm}

The theorem above formally implies Theorem \ref{thm:nilp} and Corollary \ref{cor:isoc}, but we do not know how to prove it directly without passing through the aforementioned results.
\begin{remark}
The terms isocrystal with Frobenius structure and $F$-isocrystal will be used interchangeably. Furthermore, we remark that the notion considered here differs slightly from the one found in some of the standard texts. What we call an $F$-isocrystal other authors would refer to as $F^f$-isocrystal, where $f > 0$ is a positive integer. The category of $F$-isocrystals considered here is the union of the categories of $F^f$-isocrystals for all positive integers $f$. Our usage of the term is consistent with recent advances on companions (see for example \cite{Abe18}) putting these more general $F$-isocrystals in relation with $\bar{\mathbb{Q}}_{\ell}$-adic sheaves.
\end{remark}

In fact, our statement is slightly stronger than Theorem \ref{thm:Fstr}. The isocrystal above is induced by a filtered Frobenius crystal (with a $W(\mathbb{F}_{p^f})$-endomorphism structure).
This shows that the isocrystals with Frobenius structure stemming from Theorem \ref{thm:Fstr} are associated to a crystalline representation  $\pi_1(X_{K_v})\to GL(W(\F_{p^f}))$ for some positive integer $f\ge 1$ called the \emph{period}. These statements are consequences of the theory of Lan--Sheng--Zuo, see Remark \ref{rmk:flf}.  By comparing their construction with Faltings's $p$-adic Simpson correspondence, one shows that this representation is rigid over $\pi_1(X_{\bar K_v})$, see  Theorem \ref{thm:monodromy}.  This enables one to prove that the induced projective connection is defined over $\Fb_p$ {for infinitely many primes $p$} (see also Corollary \ref{cor:f_p_one}). The proof of the third main theorem is based on this observation. This is the content of Section \ref{sec:unitary}.

\begin{thm} \label{thm:unitary}
Let $X$ be a smooth connected  projective complex variety, and let $(E,\nabla)$ be a rigid flat connection on $X$.  Assume that we have a scheme $S$ as in Theorem \ref{thm:Fstr} such that the $p$-curvature for all closed points $s$ of $S$ is $0$. Then $(E,\nabla)$ has unitary monodromy.
\end{thm}

We consider this result to be a first step towards an understanding of the Grothendieck--Katz $p$-curvature conjecture for rigid flat connections.

\medskip

An irreducible flat connection $(E,\nabla)$ with torsion determinant $L$ is called \emph{cohomologically rigid}, if $[(E,\nabla)]$ is a reduced isolated point of $\MdR(X,L,r)$. This is equivalent to vanishing of
$$H^1_{dR}(X,(\Eend(E),\nabla)) = 0,$$
and hence explains the nomenclature.

{A remark is in order concerning the rationale behind the inclusion of Section \ref{sec:comp}}. In a preliminary version of this article circulated as a preprint, we used Theorem \ref{thm:Fstr}, in combination with the theory of $p$-to-$\ell$-companions, to prove Simpson's integrality conjecture \cite{Sim92} for cohomologically rigid flat connections. 
In the meantime we found a purely Betti to $\ell$-adic argument which can be found in the short companion note \cite{EG17}. Our original strategy was based on the observation that a cohomologically rigid flat connections brings forth an $F$-isocrystal (by means of Theorem \ref{thm:Fstr}) having a complete set of companions. Since it is unknown if this property holds for arbitrary $F$-isocrystals, we decided to record this result in Section \ref{sec:comp} as Theorem \ref{thm:comp}.

\medskip
{\bf Conventions.} For a scheme $S$, we denote by $|S|$ the underlying topological space or set of points. The terminology \emph{arithmetic scheme} refers to a scheme of finite type over $\Spec \Zb$. The term \emph{variety} refers to a separated and reduced scheme of finite type over a field.

\medskip
{\bf Acknowledgement.} The article rests on Simpson's theory,  his definition and study of complex  rigid integrable connections, his program. It is  a great pleasure to thank him for the discussions we had  when we formulated 
the first step of the \emph{Leitfaden} of this article, namely when we proposed to study the $p$-curvatures of rigid connections and prove that they are nilpotent (see Theorem \ref{thm:nilp}). 
We thank him for his warm encouragements. 

Kang Zuo's explanations of the theory of Higgs-de Rham flows which he developed jointly with collaborators, proved to be indispensable to our article. It is a pleasure to thank him for his comments, and to emphasize that
Theorem \ref{thm:Fstr} is an application of the methods of \cite{LSZ13}. 
Furthermore, we profited from numerous enlightening discussions with Tomoyuki Abe about isocrystals with Frobenius structure.

Carlos Simpson and Mark Kisin pointed out errors in a very early version of Proposition \ref{prop:pcurv}. It is a pleasure to thank them for bringing this to our attention, and also for many interesting conversations on the subject.
We thank Adrian Langer for his interest, for exchanges on our wish to prove nilpotency of the $p$-curvatures of rigid connections and for discussions around the topic presented here. His beautiful lecture in Berlin on \cite{LS16} awoke our interest.  In addition to the above, we greatly benefitted from conversations with Marco d'Addezio, Alexander Beilinson, Oliver Br\"aunling, Yohan Brunebarbe, Javier Fres\'an, Ph\`ung H\^o Hai,  Moritz Kerz, Bruno Klingler, Jo\~ao Pedro dos Santos, Simon P\'epin Lehalleur, Takeshi Saito during the preparation of this work. We thank Yun Hao for pointing out a mistake in the formulation of a previous version of Theorem \ref{thm:BNR}. The first author thanks the Vietnam Institute for Advanced Study in Mathematics for hospitality, and all the participants at the workshop in Tuan Chau where a tiny piece of the material of this article has been presented. 

The anonymous referees went out of their way and above and beyond the call of duty. We believe that their helpful remarks and suggestions greatly improved the presentation of the material.  

\section{Preliminaries}\label{sec:preliminaries}

We begin by giving an overview of the theory of Higgs bundles and non-abelian Hodge theory in Subsection \ref{sub:higgs_nilp}. Subsection \ref{sub:p-curvature} aims to give an introduction to the theory of flat connections over perfect fields of positive characteristic, with a particular focus on the concept of \emph{$p$-curvature}. 
A result of Ogus--Vologodsky relates flat connections with nilpotent $p$-curvature and nilpotent Higgs bundles. Subsection \ref{sub:OV} summarises this correspondence. A useful tool in the study of Higgs bundles and flat connections in positive characteristic is the BNR correspondence which we recall in Subsection \ref{sub:bnr}. We conclude this section with a brief overview of crystals.

\subsection{Recollection on Higgs bundles and non-abelian Hodge theory}\label{sub:higgs_nilp}

Let $Z/k$ be a smooth projective variety where $k$ is an algebraically closed field. A \emph{Higgs bundle} on $Z$ is a pair $(V,\theta)$ where $V$ is a vector bundle and $\theta\colon V \to V \otimes \Omega_X^1$ is an $\Oo$-linear map satisfying the integrality condition $\theta \wedge \theta = 0$. This definition is reminiscent of a flat connection, that is, a pair $(E,\nabla)$, where $E$ is a vector bundle with a connection $\nabla$ satisfying the integrality condition $\nabla^2 = 0$. A flat connection satisfies the Leibniz rule $\nabla(fs) = f \nabla s + s df$, while a Higgs field $\theta$ is $\Oo$-linear.

If $k$ is the field of complex numbers, non-abelian Hodge theory \cite{Sim97} relates the so-called \emph{Betti}, \emph{de Rham} and \emph{Dolbeault} moduli spaces by real-analytic isomorphisms:
\[
\xymatrix{
\MDol(Z,r) \ar@{<->}[rr]^{\mathbb{R}-analytic} \ar@{<->}[rd]_-[@!-27]{\scriptstyle \mathbb{R}-analytic} & & \MdR(Z,r) \ar@{<->}[ld]^-[@!31]{\scriptstyle \mathbb{C}-analytic} \\
& \MB(Z,r). & 
}
\]
The subscript $B$ refers to the \emph{Betti space}, the moduli space of irreducible representations of the topological fundamental group $\pi_1^{\rm{top}}(Z)$ of $Z$. The construction of this quasi-affine moduli space is an application of geometric invariant theory, relying on the fact that $\pi_1^{\rm top}$ is a finitely presented group. Details of the construction can be found in \cite[Section 6]{Sim94b}.

The existence of a quasi-projective moduli space of stable Higgs bundles $\MDol(Z,r)$, as well as a quasi-projective moduli space $\MdR(Z,r)$ of stable flat connections, are theorems: for $k$ of characteristic $0$ this is a consequence of Simpson's \cite[Theorem 4.7]{Sim94}, for positive characteristic fields we refer the reader to \cite[Theorem 1.1]{Lan14}. In fact the latter also applies to base schemes of mixed characteristic (which are of finite type over a universally Japanese ring). Later on (Subsection \ref{sub:model}) we will exploit this added generality when producing \emph{arithmetic models} for the moduli spaces $\MDol$ and $\MdR$.

The moduli space of Higgs bundles $\MDol$ is particularly rich in structure. It carries a $\G_m$-action, which on the level of the moduli problem corresponds to scaling the Higgs field: 
$$\lambda \cdot (V,\theta) = (V,\lambda \theta), \; \; \; \; (\lambda \in \G_m).$$

There is a natural morphism (the \emph{Hitchin morphism}) $$\MDol(X) \to \A$$ to an affine space $\A$ called the \emph{Hitchin base} (see \cite[p. 17]{Sim94b}). On the level of the moduli problem it is given by computing (the coefficients of) the characteristic polynomial of $\theta$ which are symmetric forms on $X$, that is, global sections of $\Sym^i\Omega_X^1$.

A \emph{rigid} stable Higgs bundle $(V,\theta)$ is a Higgs bundle with torsion determinant $\sL=\det(V)$ and ${\rm trace}(\theta)=0$, 
 which induces an isolated point of the moduli space {$\MDol(X, (\sL,0),r)$}. The following lemma is due to Simpson (see \cite[Section 5]{Sim97}) and can be seen as the first step in the proof of Simpson's result that rigid representations of the fundamental group give rise to complex variations of Hodge structure \cite[Lemma 4.5]{Sim97}.

\begin{lemma}\label{nilpotent} 
If $(V,\theta)$ is a rigid stable Higgs bundle, then $\theta$ is nilpotent.
\end{lemma}

\begin{proof}
Assume that $\theta$ is not nilpotent. Then the corresponding value of the Hitchin map, that is, the characteristic polynomial $a = \chi(\theta)$ of $\theta$, is non-zero. 

We consider the $\G_m$-family of stable Higgs bundles, given by $(V,\lambda \theta)$. Since the $\G_m$-action on the space of characteristic polynomials has positive weights, and $a=\chi(\theta)$ is non-zero, we obtain a non-trivial deformation of characteristic polynomials. Therefore, the $\G_m$-family $(V,\lambda \theta)$ is a non-trivial deformation. This contradicts rigidity of $(V,\theta)$.
\end{proof}

Recall that theorem \ref{thm:nilp} asserts that mod $p$ reductions of rigid flat connections have nilpotent $p$-curvature, for $p$ sufficiently large. A possible approach to proving this result is to apply similar ideas to flat connections. This is complicated by the fact that a multiple $\lambda \cdot{} \nabla$ of a flat connection does not satisfy the Leibniz rule. However, there is still a way to make sense of the Hitchin map, and to construct deformations (non-canonically, and only of finite order) above the $\G_m$-action on the Hitchin base $\Ab$ (which has positive weights). This approach is the content of the appendix to this paper. The proof given in Subsection \ref{proof:nilp} is based on a slightly different strategy.


\subsection{Flat connections in positive characteristic and $p$-curvature}\label{sub:p-curvature}

In the following we denote by $k$ a perfect field of characteristic $p > 0$ and by $Z/k$ a smooth $k$-scheme.
Recall that $\Omega^i_Z$ refers to the $i$-th exterior power of the sheaf of K\"ahler differentials $\Omega^1_{Z/k}$.
 
A \emph{connection} on a vector bundle $E/Z$ is given by a $k$-linear map of sheaves
$$\nabla\colon E \to E \otimes \Omega_Z^1,$$
satisfying the \emph{Leibniz rule}
$$\nabla(fs) = f\nabla(s) + s \otimes df,$$
for locally defined sections $s \in E(U)$, $f \in \Oo_Z(U)$, where $U \subset Z$ is an open subset. We say that $\nabla$ is \emph{flat} or \emph{integrable}, if 
$$\nabla^2 = 0\colon E \to E \otimes \Omega_Z^2.$$

The concept of $p$-curvature, which we will describe in the following, marks the crossroads where flat connections over fields of vanishing and positive characteristic diverge. At first we need to introduce some notation essential to the definition of $p$-curvature.

Let $U \subset Z$ be a open subscheme, and $\partial \in \Theta_Z(U)$ a section of $k$-derivations of $\Oo_X$ (that is, a tangent vector field on $Z$). We denote by $\partial^{[p]}$ {the $k$-derivation of $\Oo_U$ which sends a local section $f \in \Oo(V)$ for $V \subset U$ a Zariski open to}
$$\partial^{[p]}(f) = \partial^p(f) = (\partial \cdots \partial)(f).$$
{The proof of the lemma below is based on an elementary computation involving the general Leibniz rule. We omit the details.}

\begin{lemma}
The $k$-linear endomorphism $\partial^{[p]}$ of $\Oo_U$ is a derivation. That is, it gives rise to a tangent vector field $\partial^{[p]} \in \Theta_Z(U)$.
\end{lemma}

The operation $\partial \mapsto \partial^{[p]}$ defines on the Lie algebroid $\Theta_Z$ a so-called \emph{$p$-restricted structure}. Just like usual curvature measures the discrepancy of a connection $\nabla\colon \Theta_Z \to \Eend_k(E)$ to be a map of sheaves of Lie algebras, $p$-curvature captures the extent to which a flat connection is compatible with the $p$-restricted structure.

\begin{definition}\label{defi:p-curvature}
The \emph{$p$-curvature} of a flat connection $\nabla$ on a quasi-coherent sheaf $\mathcal{F}$ on $Z$ is defined to be the $k$-linear map of sheaves
$$\psi(\nabla)\colon \Theta_Z \to \Eend_k(\mathcal{F})$$
sending a local section $\nabla \in \Theta_Z(U)$ to the $k$-linear endomorphism of $\mathcal{F}$ given by 
$$(\nabla_{\partial})^p - \nabla_{\partial^{[p]}}.$$
\end{definition}


A priori the $p$-curvature $\psi(\nabla)(\partial)$ of a flat connection is a $k$-linear endomorphism of $E$. A result of Katz describes the dependence of this endomorphism on the tangent vector field $\partial$ (see \cite[Proposition 5.2]{Kat70}).

\begin{proposition}[Katz]\label{prop:Katz}
The map of sheaves $ \psi(\nabla): \Theta_Z \to \Eend_k(E)$ is $p$-linear. That is, for a Zariski open subset $U \subset X$ and local sections $\partial_1,\partial_2 \in \Theta_Z(U)$ and $f \in \Oo_Z(U)$ we have
$$\psi(\nabla)(\partial_1 + f \partial_2) = \psi(\nabla)(\partial_1) + f^p \psi(\nabla)(\partial_2).$$
Furthermore, for every local section $\partial \in \Theta_X(U)$, the induced map of sheaves $\psi(\nabla)(\partial)\colon E|_U \to E|_U$ is $\Oo_U$-linear.
\end{proposition}

We now recall the definition of Frobenius twists and the relative Frobenius morphism. Let us denote by $w\colon k \to k$ the \emph{{arithmetic} Frobenius}, that is, the field endomorphism given by the map $\lambda \mapsto \lambda^p$ (for $\lambda \in k$).
\begin{definition}\label{defi:Frobenius_twist}
The \emph{Frobenius twist} of a $k$-scheme $Z$ is defined to be the base change 
$$Z'= Z \times_{\Spec k, w} \Spec k.$$
We denote the projection $Z' = Z \times_{\Spec k, w} \Spec k \to Z$ by $w_Z$.
\end{definition}

For every scheme $Z$ of characteristic $p$ (that is, the natural map $Z \to \Spec \Zb$ factors through $\Spec \Fb_p$) one defines the \emph{absolute Frobenius} as the morphism
$f_Z\colon Z \to Z$
which is given by the identity map on the underlying topological space $|Z|$, and by the map of sheaves of rings 
$g \mapsto g^p\text{, where }g \in \Oo_Z(U),\text{ and }U \subset Z\text{ is open}.$

For every $k$-scheme $Z$ there exists a unique morphism of $k$-schemes $F_Z\colon Z \to Z'$ such that 
\begin{equation}\label{relativeF}
f_Z = w_Z \circ F_Z
\end{equation} 
We refer to it as the \emph{relative Frobenius of $Z$}. If there is no risk of confusion, we will denote $F_Z$ by $F$.

\begin{lemma-definition} \label{lem:defn}
The $p$-curvature of a flat connection $(E,\nabla)/Z$ gives rise to an $\Oo_{Z'}$-linear map 
$$\psi(\nabla)\colon F_{Z,*}E \to F_{Z,*}E \otimes_{\Oo_{Z'} } \Omega_{Z'}^1.$$
\end{lemma-definition}

\begin{proof}
By base change we have $w_Z^* \Theta_Z=\Theta_{Z'}$. Thus, Proposition~\ref{prop:Katz} implies that $\psi(\nabla)$ 
factors as $\Theta_{Z'} \to  \Eend_{k}(E)$ and further as 
$F_Z^* \Theta_{Z'} \to  \Eend_{\Oo_Z}(E)$.  The latter is rewritten as a $\Oo_Z$-linear map
$E\to E\otimes_{\Oo_Z} F_Z^*\Omega^1_{Z'}$ which by the projection formula is equivalent to Lemma-Definition \ref{lem:defn}. 
\end{proof}

The $p$-curvature of a flat connection $(E,\nabla)$ is said to be \emph{nilpotent}, if there exists a positive integer $N$ such that
$$\psi(\nabla)^N\colon F_*E \to F_*E \otimes (\Omega_{Z'}^1)^{\otimes N}$$
is the zero map. According to a theorem of Katz (see \cite[Theorem 5.10]{Kat70}), Gau\ss-Manin connections have nilpotent $p$-curvature. 
\begin{theorem}[Katz]\label{thm:katz}
Let $f\colon Y \to Z$ be a smooth projective morphism between smooth $k$-varieties. Then the flat connection given by the  Gau\ss-Manin connection on $R^i f_* (\Oo_Y,d)$ has nilpotent $p$-curvature.
\end{theorem}

Katz's result therefore turns $\psi(\nabla)$ into an invariant which can be used to disprove that a given connection $(E,\nabla)$ on $Z$ is of \emph{geometric origin} (that is, a subquotient of a Gau\ss-Manin connection). If $\psi(\nabla)$ is not nilpotent, then $(E,\nabla)$ does not stand a chance of being of geometric origin. 

\medskip

The constructions reviewed in this subsection are also defined in a relative set-up. We briefly summarise the main points. 

We denote by $T$ a $k$-scheme and let $Z \to T$ be a smooth morphism.
Recall that the absolute Frobenius morphism of $T$ is denoted by $f_T\colon T \to T$. One defines the \emph{relative Frobenius twist} to be the base change $Z' = Z \times_{T,f_T}T$. By virtue of definition it is a $T$-scheme. There is a unique morphism of $T$-schemes $F_{Z/T}\colon Z \to Z'$, called \emph{relative Frobenius morphism} which satisfies
$$f_Z = w_Z \circ F_{Z/T}.$$ 
If there is no risk of confusion, we will denote $F_{Z/T}$ by $F$.

A \emph{de Rham sheaf} on $Z/T$ is a pair $(E,\nabla)$ where $E$ is a quasi-coherent sheaf on $Z$ and $\nabla\colon E \to E \otimes \Omega_{Z/T}^1$ is  an integrable connection, that is, an $\Oo_T$-linear morphism which satisfies the Leibniz rule and the flatness condition $\nabla^2 = 0$. The $p$-curvature of a flat connection $\nabla$ as defined in \cite[5.0.4]{Kat70} will be referred to as $\psi(\nabla)\colon E \to   E\otimes F_{Z/T}^*\Omega_{Z'/T}^1 $. 


\subsection{PD differential operators and Azumaya algebras}
{
A flat connection $\nabla$ on a quasi-coherent sheaf $E$ on a smooth $k$-variety $Z$ induces the structure of a $D_Z$-module on $E$, where $D_Z$ denotes the sheaf of rings of PD differential operators defined below. This is analogous to the fact that a representation of a Lie algebra $\mathfrak{g}$ amounts to a $U\mathfrak{g}$-module, where $U\mathfrak{g}$ denotes the universal enveloping algebra. In positive characteristic, the sheaf of rings $D_Z$ has a large centre, over which it defines an Azumaya algebra. In the following we will describe this observation, which appeared first in Bezrukavnikov--Mirkovic--Rumynin \cite{BMR08}, in more detail.

For a quasi-coherent sheaf $M$ on $Z$ we use the notation $T^{\bullet}M$ to denote the sheaf of tensor algebras 
$$T^{\bullet} M = \bigoplus_{n \geq 0} M^{\otimes n},$$
we write $\Theta_{Z}$ for the sheaf of tangent vectors.}

\begin{definition}
The sheaf of algebras $D_{Z}$ is defined to be the sheafification of $T^{\bullet} \Theta_{Z}$ modulo the relations 
\begin{equation}\partial \cdot f- f\cdot \partial=\partial(f)\end{equation}
\begin{equation}\partial \otimes \partial' - \partial' \otimes \partial = [\partial,\partial']\end{equation}
for local sections $\partial$, $\partial'$ of $\Theta_{Z}$ and $f$ of $\Oo_Z$.
\end{definition}

{The same ideas underlying the $p$-curvature give rise to a map $\psi\colon \Theta_{Z'} \to F_*D_Z$, which sends $\partial$ to $\partial^p - \partial^{[p]}$. }

\begin{proposition}[Bezrukavnikov--Mirkovic--Rumynin]\label{prop:bmr}
{If $k$ is a perfect field of positive characteristic, then the map of sheaves of algebras $\psi\colon \Sym^{\bullet} \Theta_{Z'} \to F_* D_Z$ is an injection whose image agrees with the centre of $F_*D_{Z}$.}
\end{proposition}

Recall that the total space of a vector bundle $V$ on a scheme $Y$ is the scheme given by the relative spectrum
$$\mathrm{Tot}\; V = \Spec_Z \Sym^{\bullet} V^{\vee}.$$
The sheaf of symmetric algebras $\Sym^{\bullet} \Theta_{Z'}$, arising in the proposition above, is therefore isomorphic to $\pi_*\Oo_{T^*Z'}$, where $\pi$ denotes the canonical projection $T^*Z' \to Z'$. 

\begin{lemma-definition}\label{lemma-definition} There exists a quasi-coherent sheaf of algebras $\sD_{Z}$ on $T^*Z'$ such that 
$\pi_*\sD_{Z} \simeq F_*D_Z$.
\end{lemma-definition}

\begin{proof}
For every affine morphism of schemes $f\colon W \to Y$ one has an equivalence of categories 
$$f_*\colon \mathsf{QCoh}_W(\Oo_W) \cong \mathsf{QCoh}_Y(f_*\Oo_W)$$
between quasi-coherent sheaves of $\Oo_W$-modules and quasi-coherent sheaves of $f_*\Oo_W$-modules. We apply this observation to $\pi\colon T^*Z' \to Z'$ and the quasi-coherent sheaf of algebras $F_*D_Z$. {Since it is a $Z(F_*D_Z)$-module, and the centre $Z(F_*D_Z)$ can be identified with the quasi-coherent sheaf of algebras $\pi_*\Oo_{T^*Z'} = \Sym^{\bullet} \Theta_{Z'}$ by virtue of Proposition \ref{prop:bmr}, we conclude that there exists a quasi-coherent sheaf $\sD_{Z}$ on $T^*Z'$, such that }
\begin{equation}\label{eqn:piFDZ}
\pi_*\sD \simeq F_*D_Z.\end{equation} 
Furthermore, $F_*D_Z$ is a sheaf of $Z(F_*D_Z)$-algebras. We infer that $\sD_Z$ inherits a canonical structure of a sheaf of algebras such that \eqref{eqn:piFDZ} is in fact an isomorphism of quasi-coherent sheaves of algebras.
\end{proof}

The relative Frobenius morphism $F\colon Z \to Z'$ is also affine. For the same reasons as in the proof of Lemma-Definition \ref{lemma-definition}, we have an equivalence of categories
$$F_*\colon \mathsf{QCoh}(D_Z) \cong \mathsf{QCoh}(F_*D_Z).$$
Applying Lemma-Definition \ref{lemma-definition}, we see that the right hand side is equivalent to the category $\mathsf{QCoh}(\pi_*\sD_Z)$. Since $\pi$ is an affine morphism, we obtain the following 

\begin{lemma}\label{lemma:equivalence}
There is an equivalence of categories
$$\mathsf{QCoh}(D_Z) \cong \mathsf{QCoh}(\sD_Z).$$
\end{lemma}

This lemma allows us to describe quasi-coherent sheaves with flat connections on $Z$ in terms of quasi-coherent $\sD$-modules on $T^*Z'$. According to a result of \cite{BMR08}, the algebra $\sD$ is \emph{Azumaya}. 

\begin{theorem}[Bezrukavnikov--Mirkovic--Rumynin]
Assume that $Z$ is pure dimensional of dimension $d$. The sheaf of algebras $\sD$ of Lemma-Definition \ref{lemma-definition} is an Azumaya algebra of rank $p^{2d}$, that is, there exists an \'etale covering $\{U_i \xrightarrow{f_i} T^*Z'\}_{i \in I}$ such that we have $$f_i^*\sD \simeq \Eend(\Oo_{U_i}^{p^d}).$$
\end{theorem}

%
%
%

\subsection{The Ogus--Vologodsky correspondence}\label{sub:OV}

As before we denote by $k$ a perfect field of positive characteristic $p$ and by $Z/k$ a smooth $k$-scheme. The pullback of a quasi-coherent sheaf $V$ along the relative Frobenius $F\colon Z \to Z'$ is endowed with a \emph{canonical connection} $\nabla^{can}$. It is uniquely characterised by the property that a local section $s \in F^*V(U)$ is $\nabla^{can}$-horizontal (that is, satisfies $\nabla^{can}(s) = 0$) if and only if $s \in F^{-1}V(U)$. We refer the reader to \cite[Theorem 5.1]{Kat70} for a proof of the following

\begin{theorem}[Cartier descent]
The functor $V \mapsto (F^*V,\nabla^{can})$ embeds the category $\mathsf{QCoh}(Z')$ into $\MIC(Z)$. A de Rham sheaf $(E,\nabla)$ is isomorphic to $(F^*V,\nabla^{can})$ if and only if $\nabla$ has zero $p$-curvature.
\end{theorem}

An important result of Ogus--Vologodsky extends this embedding to certain nilpotent Higgs bundles on $Z'$. At first we need to introduce some notation.

\begin{definition}\label{defi:MICZ}
\begin{itemize}
\item[(a)] We say that $\theta$ is nilpotent of level $\leq N$ if the induced morphism $\theta^N\colon V \to V \otimes (\Omega_{Z}^1)^{\otimes N}$ is the zero morphism. We denote the resulting category by $\Higgs_N(Z)$.
\item[(b)] For a positive integer $N$ we denote by $\MIC_N(Z)$ the category of de Rham sheaves $(E,\nabla)$, where $\psi(\nabla)$ is nilpotent of level $\leq N$, that is, the map $\psi^N\colon E \to E \otimes (F_{Z}^*\Omega_{Z'}^1)^{\otimes N}$ is zero.
\end{itemize}
\end{definition}

We refer the reader to \cite[Theorem 2.8]{OV07} for a proof of the theorem below and a more detailed overview of this result. We remark that the Ogus--Vologodsky-correspondence is deduced by producing splittings of the Azumaya algebra $\sD$ over infinitesimal thickenings of the zero section $Z' \hookrightarrow T^*Z'$.

\begin{theorem}[Ogus--Vologodsky]\label{thm:OV} 
{We use the terminology introduced in Definition \ref{defi:MICZ}. A lifting $\mathcal{Z} \to \Spec W_2(k)$ of $Z \to \Spec k$ gives rise to an equivalence of categories}
$$C^{-1}_{\mathcal{Z}/W_2(k)}\colon \Higgs_{p-1}(Z') \cong \MIC_{p-1}(Z),$$ 
such that $C^{-1}_{\mathcal{Z}/W_2(k)}(V,0) \simeq (F^*V,\nabla^{can})$.
\end{theorem}

\subsection{The Beauville--Narasimhan--Ramanan correspondence}\label{sub:bnr}

{In the following we fix a line bundle $L'$ on $Z'$ of {finite order} invertible in $k$. Its pullback {$F_Z^*L'$}  along the relative Frobenius map is isomorphic to $L^p$. The notion of Gieseker stability (also known as $P$-stability) allows one to construct a quasi-projective coarse moduli space of $P$-stable integrable connections $\MdR(Z/k, L^p,r)$ with determinant $L^p$.
We refer the reader to \cite[Theorem 1.1]{Lan14} for the notion of $P$-stability and for details on the construction of the moduli space.  

The $p$-curvature  {$\psi(\nabla)\in H^0(Z, F_Z^*\Omega^1_{Z'} \otimes \Eend(E))$}  of any integrable connection $(E,\nabla)$ on $Z$ is flat under the tensor product of the canonical connection on $F^*\Omega^1_{Z'}$ and of $\Eend(\nabla)$ on $\Eend(E)$ (\cite[Proposition 5.2.3]{Kat70}). In particular its characteristic polynomial has coefficients in global symmetric forms on $Z'$:
\ml{1}{\chi(\psi(\nabla))={\rm det}(-\psi(\nabla) +\lambda {\rm Id})= \lambda ^r -a_{1} \lambda^{r-1} +  \ldots + (-1)^r a_r, \\ a_i \in H^0(Z', {\rm Sym}^{i}(\Omega^1_{Z'})).}

It was observed by Laszlo--Pauly \cite[Proposition 3.2]{LP01} that $\chi(\psi(\nabla))$ gives rise to a morphism,  called \emph{Hitchin map}
\ga{2}{ \chi_{dR}: \MdR(Z/k,r)\to \A_{Z',r},}
where the affine space $\A_{Z',r}$ is given by
\ga{3}{\A_{Z',r}(T)= \bigoplus_{i=2}^r H^0(Z', {\rm Sym}^{i}(\Omega^1_{Z'}))\otimes_{k} T}
for any $k$-algebra $T$. See also the work of Bezrukavnikov--Braverman \cite[Section 4]{BB07}, the second author \cite[Definitions 3.12 \& 3.16]{Gro16} and Chen--Zhu \cite[Section 2.1]{CZ15}. Properness of this map was established in \cite[Theorem 3.8]{Lan14}.

{Note that the definition above omits the $i=1$ term, since our Higgs fields are tracefree by assumption.}

It follows from the definitions that $\psi(\nabla)$ is nilpotent if and only if $\chi_{dR}([(E,\nabla)])=0$,  where  $[(E,\nabla)]$ is the moduli point of the connection $(E,\nabla)$. 

\medskip

Before recalling the Beauville--Narasimhan--Ramanan (BNR) correspondence for flat connections proven in \cite[Proposition 3.15]{Gro16} we need to define spectral covers. For any $k$-scheme $S$ and $a: S\to \A_{Z',r}$ one has a finite cover
{
\ga{}{ \xymatrix{
\ar[dr]_{\pi'} Z'_a \ar@{^(->}[r] & T^*Z' \times S\ar[d]^{\pi'} \\
 & Z} \notag}
}
defined by the equation $\eqref{1}=0$.  This cover will be referred to as the \emph{spectral cover}.

\begin{remark}  \label{rmk:spec}
Since it is central to our approach, {we give a more detailed description of the definition of the spectral cover.} Consider the quasi-coherent sheaf of algebras given by symmetric forms $\Sym^{\bullet} \Omega_{Z'}^1$ with its natural grading. We adjoin a formal variable $\lambda$ of degree $1$ and obtain the quasi-coherent sheaf of algebras $\Sym^{\bullet} \Omega_{Z'}^1[\lambda]$. Recall that the points of the Hitchin base $\A_{Z',r}$ 
are defined by \eqref{3}. The following more detailed description is useful.
\begin{itemize}
\item[(a)] The vector space of degree $r$  sections  in  $H^0(Z',\Sym^{\bullet} \Omega_{Z'}^1[\lambda])$ is isomorphic to $\A_{Z',r}(k)$. \end{itemize}
The tautological section $\theta \in H^0(Z', \pi^{'*}\Omega^1_{Z'})$ is defined by the  $\sO_{Z'}$-linear homomorphism {
$ \sO_{Z'}\to \pi'_* \pi'^*\Omega^1_{Z'} =\Sym^{\bullet}  T_{Z'}  \otimes _{\sO_{Z'}} \Omega^1_{Z'}  $ }  which is equal to  the identity on the factor {$T^*_{Z'}  \otimes _{\sO_{Z'}} \Omega^1_{Z'} = \mathcal{E}nd (\Omega^1_{Z'}) $}
 and equal to zero on the other factors. 
Pullback along $\pi'$ for $S= \Spec k$ postcomposed with the specialization $\lambda \mapsto \theta$ defines a morphism of algebras 
$$H^0(Z',\Sym^{\bullet} \Omega_{Z'}^1[\lambda]) \to H^0(T^*Z',\pi'^*\Sym^{\bullet} \Omega_{Z'}^1).$$
Similarly, one obtains for a $k$-scheme $S$ a morphism
$$H^0(Z' \times_k S,\Sym^{\bullet} \Omega_{Z' \times_k S/S}^1[\lambda]) \to H^0(T^*Z' \times_k S, \pi'^*\Sym^{\bullet} \Omega_{Z' \times_k S/S}^1).$$
\begin{itemize}
\item[(b)] For $a \in \A_{Z',r}(S)$ we define the spectral cover $Z'_a$ to be the closed subscheme of  $T^*Z' \times_k S$  defined by the sheaf of ideals generated by the section $ \theta^r  +a_2 \theta^{r-2} + \cdots + (-1)^r a_r$
in $(\pi' \times_k {\rm Id})_* \sO_{T^*{Z'} \times_k S}=(\Sym^\bullet T_{Z'} )\otimes_k \sO_S$. Then {$Z'_a\to Z'\times_k S$} is a finite morphism.
\end{itemize}
\end{remark}

It is called the {\it spectral cover} $Z'$ to $a$.  We can now state the BNR correspondence for flat connections. In the following denote by $S$ a $k$-scheme and let $a \in \A_{Z',r}(S)$. An $S$-family of integrable connections on $Z$ refers to a pair $(E,\nabla)$ where $E$ is vector bundle on $Z \times_k S$ and 
$$\nabla\colon E \to E \otimes {\rm pr}_Z^*\Omega^1_{Z}$$
is an $\Oo_S$-linear map satisfying the Leibniz rule and $\nabla^2 = 0$.
\begin{thm}[BNR correspondence] \label{thm:BNR}
The groupoid of $S$-families of rank $r$ integrable connections $(E,\nabla)$ ({with $E$ being a vector bundle}) on $Z\times_k S$ satisfying $\psi(\nabla) = a$, is equivalent to the groupoid of $\sD_{Z'}$-modules $M$ on $Z'_a\subset T^*Z' \times_k S$ such that $\pi'_*M$ is a locally free Higgs sheaf on $Z'$ of rank $p^{d}r$ and characteristic polynomial $a^{p^d}$. 
\end{thm}

\begin{remark}
Recall that we have a quasi-coherent sheaf of $\Oo_{T^*Z'}$-algebras $\sD_{Z'}$. We say that  on $T^*_{Z'}\times_k S$ a quasi-coherent $p_{T^*_{Z'}}^*\sD_{Z'}$-module $M$ is {\it scheme-theoretically supported on a closed subscheme} $i \colon X \hookrightarrow T^*Z' \times_k S $, if $M$ is annihilated by the sheaf of ideals $\mathcal{I}_X \subset \Oo_{T^*_{Z'}\times_k S}$ of $X$. The scheme-theoretic support of $M$ as a $p_{T^*_{Z'}}^*\sD_{Z'}$-module only depends on $M$ as a quasi-coherent sheaf on $T^*Z'\times_k S$. The main part of the proof below is devoted to improving an a priori bound for the scheme-theoretic support.
\end{remark}

In order for our article to remain as self-contained as possible, we recall the proof from \cite[Proposition 3.15]{Gro16}. 

\begin{proof}[Proof of Theorem \ref{thm:BNR}] 
The connection $\nabla$ defines the structure of a $p_{Z}^* D_Z$-module on $E$ over $Z\times_k S$. Therefore we obtain a $F_*p_{Z}^* D_Z$-module $F_*E$, which can be written as pushforward of a $ p_{T^*_{Z'}}^*\sD_{Z'}$-module $M$ on $T^*Z' \times_k S$ along the canonical projection $\pi'\colon T^*Z' \times_k S \to Z' \times_k S$ (see Lemma \ref{lemma:equivalence}). As remarked above, the scheme-theoretic support of $M$ depends only on its $\sO_{T^*Z' \times_k S}$-module structure, which is induced by the $p$-curvature $\psi(\nabla)\colon F_*E \to   p_{Z'}^*\Omega^1_{Z'} \otimes_{\sO_{Z' \times_k S}} F_*E $. 

The pair $F_*(E,\psi(\nabla))$  is a Higgs bundle of rank $rp^d$ on $Z' \times_k S \to S$, where $d={\rm dim}(Z)$. Let $b \in \A_{Z', rp^d}(S)$ be the characteristic polynomial of the Higgs field $\psi(\nabla)$. Then the Higgs bundle 
$F_*(E,\psi(\nabla))$ 
is scheme-theoretically supported on the closed subscheme 
$$Z'_b  \subset T^*Z' \times_k S.$$ However this ``upper bound" for the scheme-theoretic support is far from being optimal. We  construct a degree $r$ polynomial $a \in \A_{Z',r}(S)$ such that $b=a^{p^d}$ (for the multiplication as polynomials), and such that $M$ is scheme-theoretically supported on $Z'_a$.
 It suffices to show that $M$ is scheme-theoretically supported on $Z'_a \hookrightarrow Z'_b \hookrightarrow  T^*Z' \times_k S$.

It is clear that there is at most one such $a \in \A_{Z',r}(S)$ with this property since $a$ and $b$ are monic polynomials. By virtue of \'etale descent for symmetric forms on $Z'$, it suffices to construct $a$ \'etale locally. For the same reasons, one only has to prove \'etale locally that $M$ is scheme-theoretically supported on $Z'_a$.

Let $x \in Z' \times_k S(k^{\text{sep}})$ be a geometric point. We denote by $\Oo_x^h$ the corresponding henselian local ring. Pulling back the finite morphism $Z'_b \to Z' \times_k S$ along $\Spec \Oo_x^h \to Z' \times_k S$ we obtain the spectrum of a product of henselian local rings $\Spec R = \Spec \prod_{i= 1}^N R_i$. Since $\sD_{Z'}$ is an Azumaya algebra, the pullback $\sD_{Z'}|_{\Spec R}$ splits, that is, is isomorphic to the sheaf of matrix algebras $M_{p^d}(\Oo_{\Spec R})$. 

Classical Morita theory (see \cite[Theorems 18.11 \& 18.2]{Lam99}) implies that every quasi-coherent $M_{p^d}(\Oo_{\Spec R})$-module is isomorphic to a unique (up to a unique isomorphism) $M_{p^d}(\Oo_{\Spec R})$-module of the shape $F \otimes_{\Oo_{\Spec R}} \Oo_{\Spec R}^{p^d} = F^{\oplus p^d}$ with the canonical $M_{p^d}(\Oo_{\Spec R})$-action.

Applying the pushforward functor $\pi'_*$ we obtain that the Higgs bundle $$(F_*(E,\psi(\nabla))|_{\Spec \Oo_x^h}$$ splits as a direct sum $\pi'_*F^{\oplus p^d}$. Let $a$ be the characteristic polynomial of the Higgs bundle $\pi'_*F$ on $\Spec \Oo_x^h$. We have $b|_{\Spec \Oo_x^h} = a^{p^d}$ (as polynomials), and $F$ is scheme-theoretically supported on $Z'_a$. Since $F^{p^d} = M$ the same also holds for $M$.

The affine scheme $\Spec \Oo_{x}^h$ is an inverse limit of \'etale neighbourhoods of $x$. From a finite presentation argument we obtain an \'etale morphism $U \xrightarrow{h} Z' \times_k S$ such that the Azumaya algebra $\sD_{Z'}$ splits already when pulled back to $U \times_{Z' \times_k S} Z'_b$. Repeating the argument above we obtain a degree $n$ section of $h^*\Sym^{\bullet} \Omega_{Z' \times S/S}[\lambda]$ (where $\lambda$ is a formal variable of weight $1$)
$$a = \lambda^r   +a_2\lambda^{r-2} +\ldots +a_r$$
such that $a^{p^d} = h^*b$.

By uniqueness of solutions to the equation $a^{p^d} = b$ in $\Sym^{\bullet} \Omega^1_{Z' \times S/ S}[\lambda]$ we obtain $a \in \A_{Z',r}(S)$ with the required property. 
We conclude from the local strict henselian support property 
 that $M$ is scheme-theoretically supported on $Z'_a$.  This finishes the proof.
\end{proof}

}

\subsection{Crystals and $p$-curvature}\label{sub:criss}

In this short subsection we collect the necessary references and facts from the theory of crystals which we use in the core of our article. We do not claim any originality for the results presented here. The main purpose is only to gather the references needed for (the well-known) Theorem \ref{thm:quasinilp} and Corollary \ref{cor:quasinilp} below.
We warmly thank Pierre Berthelot, Luc Illusie, Arthur Ogus,  and Atsushi Shiho for their kind and efficient answers to our questions on references.

Let $k$ be a perfect field, $Z$ be a  smooth $k$-variety. Let 
$W=W(k)$ be the  ring of Witt vectors on $k$,  $W_n=W/p^n$ be the ring of Witt vectors of length $n$. 
  One defines the crystalline sites $Z/W_n$  as in 
\cite[III, Definition 1.1] {Ber74}, then the crystalline site $Z/W$ as the union (or colimit) of the crystalline sites for $n\in \N_{>0}$  (\cite[1.1.3]{BBM82}) along the embeddings 
$$(Z/W_n)_{\rm crys} \hookrightarrow (Z/W_{n+1})_{\rm crys}.$$ 
By definition, the objects of $(Z/W_n)_{\rm crys}$ are relative PD-thickenings over $W_n$. That is, they are triples $(U,T,\delta)$, where 
$U \hookrightarrow Z$ is a Zariski open subset, $U \to T$ is a closed immersion of $W_n$-schemes defined by a sheaf of ideals $\mathcal{I}$, and $\delta$ is a \emph{divided power structure} on $\mathcal{I}$ compatible with the one on $pW_n$.

One defines the category ${\mathsf{Crys}}(Z/W_n)$ of crystals as the category of sheaves of $\sO$-modules $\mathcal{F}$ on $Z/W_n$, of finite type (it is also possible to work with crystals in quasi-coherent sheaves, we restrict ourselves to  $\sO$-modules of finite type), which are \emph{crystals}, that is, for every morphism $f\colon (U,T_1,\delta_a) \to (U,T_2,\delta_2)$ in the crystalline site of $Z/W_n$
one assumes that the transition map 
$$f^*\mathcal{F}_{T_2} \to \mathcal{F}_{T_1}$$
is an isomorphism. The resulting category of crystals ${\mathsf{Crys}}(Z/W_n)$  is abelian and $W_n$-linear \cite[IV 1.7.6]{Ber74} (in \emph{loc. cit.} it is shown that the bigger category of crystals in quasi-coherent sheaves is abelian, this implies the assertion with the finite type assumption, since $Z$ is assumed to be smooth). The $W$-linear  category ${\mathsf{Crys}}(Z/W)$ is defined in \cite[1.3.3]{BM90} ({for the big crystalline site, here we consider only the small one}). Its $\Q$-linearization 
$$\Isoc(Z/W) = {\mathsf{Crys}}(Z/W)_{\Q}$$ 
is the category of isocrystals (see \cite[Definition 0.7.1]{Ogu90}).

We now assume that we have a smooth formal lift $\hat Z_W$ at disposal, defining $Z_n=\hat Z_W\otimes_W W_n$. 
 By \cite[II Theorem 4.3.10, IV Theorem1.6.5]{Ber74}, the category ${\mathsf{Crys}}(Z/W_n)$ is equivalent to $\mathsf{MIC}(Z_n)^{\rm qn}$, where $\mathsf{MIC}(Z_n)$ is the category of modules  $(E_n,\nabla_n)$ of finite type with an integrable connection, and the superscript ${\rm qn}$ refers to \emph{quasi-nilpotency} defined in 
 \cite[II, Definitions 4.3.5 \& 4.3.6] {Ber74}. By \cite[Exercise 4.14]{BO78},  $(E_n,\nabla_n)$ is quasi-nilpotent if and only $(E_1,\nabla_1)=(E_n, \nabla_n)\otimes_{W_n} k$ is. Indeed, as the differential operators commute with the tensor product $\otimes_{W_n} W_m$  for $m\le n$, one direction is trivial. Vice-versa, if $(E_1,\nabla_1)$ is quasi-nilpotent, and $s_n$ is a local section of $E_n,$ then 
 some differential operator $P$ { of a certain  order}   annihilates $s_n\otimes_{W_n} W_1$ ({see  \cite[07JE]{stacks} for the precise definition}), or equivalently $P(s_n) =p \tilde{s}_{n-1}$  where 
 $s_{n-1}$ is a section of $E_{n-1}$ and $\tilde{s}_{n-1}$ any lift in $E_n$.   One iterates to conclude the proof.   
 
  Finally,  by virtue of \cite[Cor.5.5]{Kat70}, on $Z=Z_1$, a module  $(E_1,\nabla_1)$ of finite type with an integrable connection is quasi-nilpotent if and only if its $p$-curvature is nilpotent.  
This implies the following result. 
\begin{theorem} \label{thm:quasinilp}
 The category ${\mathsf{Crys}}(Z/W_n)$ is equivalent to the full subcategory of $\mathsf{MIC}(Z_n)$ consisting of the
 modules  $(E_n,\nabla_n)$ of finite type with an integrable connection, and such that the $p$-curvature of
 $(E_1,\nabla_1)$  is nilpotent. 
 \end{theorem}
The equivalence ${\mathsf{Crys}}(Z/W_n)$ with $\mathsf{MIC}(Z_n)^{\rm qn}$ induces the equivalence between ${\mathsf{Crys}}(Z/W)$ and the subcategory of modules  $(E,\nabla)$ of finite type with an integrable connection on $\hat Z_W$ which are separated and complete, that is $(E,\nabla)= \varprojlim_n
(E_n,\nabla_n)$ where $(E_n,\nabla_n)=
 (E,\nabla)\otimes_W W_n$ lies in $\mathsf{MIC}(Z_n)^{\rm qn}$, see \cite[1.3.3]{BM90}. 

\begin{corollary}\label{cor:quasinilp}
 The category ${\mathsf{Crys}}(Z/W)$ is equivalent to the full subcategory of $\mathsf{MIC}( \hat Z_W)$ consisting of the
 modules  $(E,\nabla)= \varprojlim_n
(E_n,\nabla_n)$ of finite type with an integrable connection, which are separated and complete such that the $p$-curvature of
 $(E_1,\nabla_1)$ is nilpotent. 
\end{corollary}

Let $X \to \Spec W(k)$ be a smooth morphism. We denote by $Z$ the base change $X \times_{\Spec W(k)} \Spec k$, and by $i\colon Z \to X$ the projection to the first factor. Given an object $(E,\nabla)$ of $\MIC(X/W)$, the above criterion allows us to associate to it a crystal, as long as the $p$-curvature of $\nabla$ is nilpotent.

\begin{corollary}
Assume that $i^*(E,\nabla)$ has nilpotent $p$-curvature on $X$. Then the formal flat connection $(\hat {E},\hat {\nabla})$ on $\hat X_W$ gives rise to an object of $\mathsf{Crys}(Z/W)$. 
\end{corollary}

\section{$p$-curvature and rigid flat connections} \label{sec:proofthmnilp}

This section is devoted to the study of mod $p$ reductions of rigid flat connections and culminates in a proof of our first main result, Theorem \ref{thm:nilp} in Subsection \ref{proof:nilp}. The earlier parts of this section lay the foundation for this argument. In Subsection \ref{sub:model} we discuss models of $X$ and rigid flat connections $(E,\nabla)$ over a scheme $S$ of finite type. These models will then be used in Subsection \ref{sub:cartiertransform} to study the interplay of the Ogus--Vologodsky transform and rigidity. {We recall our standing assumption that a rigid flat connection is stable (see Remark \ref{rmk:standing}).}

\subsection{Arithmetic models}\label{sub:model}

Let  $X$  be a smooth complex projective variety and $L \in \Pic(X)$ a line bundle of finite order. We define
\ga{3.1}{ \MdRrig(X/\C,L,r) \subset \MdR(X/\C,L,r)}
to be
the closed subscheme of isolated points of the quasi-projective moduli of $P$-stable integrable connections ({an isolated point is not assumed to be reduced}). As $\MdR(X/\C,L,r)$ is quasi-projective,  $\MdRrig(X/\C,L,r)$ is a $0$-dimensional quasi-projective $\mathbb{C}$-variety, and therefore projective. Furthermore, by definition $\MdR(X/\C,L,r)$ is a disjoint union of $\MdRrig(X/\C,L,r)$ and of its open and closed complement 
$$\MdR(X/\C,L,r) \setminus \MdRrig(X/\C,L,r),$$ which does not contain isolated points.  The same remarks apply mutatis mutandis to the moduli space of $P$-stable Higgs bundles $$\MDolrig(X/\C,L,r) \subset \MDol(X/\C,L,r).$$

\begin{lemma}[Arithmetic models]\label{lemma:arithmetic_models}
There exists a morphism of schemes $X_S \to S$ satisfying the following conditions: 
\begin{enumerate}
\item[(a)] $S$ is of finite type and smooth over $\Spec \mathbb{Z}$;
\item[(b)] $S$ has a unique generic point $\eta$, and there is an embedding of fields $k(\eta) \subset \Cb$;
\item[(c)] the base change along the map $\Spec \Cb \to S$ of (b) satisfies $\Spec \mathbb{C} \times_S X_S \simeq X$;
\item[(d)] the map $X_S \to S$ is smooth and projective;
\item[(e)] there is a line bundle $L_S \in \Pic(X_S)$ such that $L_S$ pulls back to a line bundle isomorphic to $L$ on $X$.
\end{enumerate}
\end{lemma}

\begin{proof}
We denote by $\mathcal{R}$ the set of subrings $R \subset \Cb$ which are of finite type over $\Zb$. Since $X$ is a projective $\Cb$-scheme there exists $R \in \mathcal{R}$ such that there is a projective $R$-scheme $X_R$ together with an isomorphism
$$X_R \times_{\Spec R} \Cb \simeq X.$$
Indeed, it suffices to choose an explicit presentation of $X \subset \Pb^N_{\Cb}$ by a system of homogeneous equations, and to consider the smallest subring $R \subset \Cb$ containing the coefficients of these homogeneous polynomials. It follows from \cite[Th\'eor\`eme 8.8.2(ii) \& 8.10.5(xiii)]{EGA} that we can even choose $X_R \to \Spec R$ to be a smooth and projective $R$-scheme. Similarly, the results \cite[Th\'eor\`eme 8.5.2(i) \& Proposition 8.5.5]{EGA} show that $X_R$ can be assumed to possess a line bundle $L_R$ which pulls back to $L$ on $X$ (up to isomorphism).
 
If $\Spec R \to \Spec \Zb$ is not already smooth, then it suffices to invert a finite number of elements $f_1,\dots,f_m$ of $R$ such that $\tilde{R} = R[f_1^{-1},\dots,f_m^{-1}] \subset \Cb$ is smooth over $\Zb$. We set $S = \Spec \tilde{R}$ and define $X_S = X_R \times_{\Spec R} S$. By construction, it satisfies all of the conditions above.
\end{proof}

A pair $(X_S,L_S)$ as above will also be referred to as an \emph{arithmetic model} of $(X,L)$. The proofs of our main results are based on a careful choice of arithmetic models.  

By \cite[Theorem 1.1]{Lan14},  there are quasi-projective moduli $S$-schemes
$$\MdR(X_S/S,  L_S,r) \to S, \ 
\MDol(X_S/S,  L_S,r) \to S.$$  For any  locally noetherian $S$-scheme $T$, one has a morphism $\varphi_{T}\colon \MdR(X_S/S,L_S,r) \times_S T\to \MdR(X_T/T, L_T,r)$. If $T$ is a geometric point, $\varphi_T$ induces an isomorphism on geometric points on both sides, and likewise for $\MDol(X_S/S,  L_S,r)$. In order to simplify notation, we will denote the disjoint union
$$\bigsqcup_{r' \leq r}\MdR(X_S/S,  L_S,r')$$
by the shorthand $\MdR(X_S/S,  L_S,\leq r)$. The same remark and notational convention applies to $\MDol(X_S/S,  L_S,r)$ mutatis mutandis.

\begin{definition}\label{defi:rig}
For an $S$-scheme $\mathcal{X} \to S$ we denote by $\mathcal{X}^{\rm rig}$ the maximal open subscheme such that $\mathcal{X}^{\rm rig} \to S$ is quasi-finite at all points of $\mathcal{X}^{\rm rig}$ (see \cite[01TI]{stacks} for a proof of openness).
\end{definition}

We apply this definition to the moduli $S$-schemes $\MdR(X_S/S,  L_S,r)$ and $\MDol(X_S/S,  L_S,r)$. The corresponding open subschemes will be denoted by $\MdRrig(X_S/S,  L_S,r)$, respectively $\MDolrig(X_S/S,  L_S,r)$.


\begin{proposition}[Nice models]\label{prop:nice_models}
For every positive integer $r$ there exists an affine arithmetic scheme $S$ and a model $(X_S,L_S)$ of $(X,L)$ such that the following conditions hold:
\begin{enumerate}
\item[(a)] For every rigid flat connection $(E_{\Cb},\nabla_{\Cb})$ over $X$ with determinant $L$ and rank $\leq r$ there exists {a spreading out} to a relative flat connection $(E_S,\nabla_S)$ on $X_S/S$ which is $P$-stable over geometric points.
\item[(b)] For every rigid Higgs bundle $(V_{\Cb},\theta)$ over $X$ with determinant $L$ and rank $\leq r$ there exists {a relative Higgs bundle} $(V_S,\theta_S)$ on $X_S/S$ which is $P$-stable over geometric points.
\item[(c)] Furthermore, {in (b) we may assume the Higgs field $\theta_S$ to be nilpotent.}
\item[(d)] The sections 
$$[E_S,\nabla_S]\colon S \to \MdR(X_S/S,  L_S,\leq r)$$ and 
$$[V_S,\theta_S]\colon S \to \MDol(X_S/S,  L_S,\leq r)$$ 
induced by $(E_S,\nabla_S)$ of (a), respectively $(V_S,\theta_S)$ of (b) and (c) factor through $\MdRrig(X_S/S,  L_S,\leq r)$, respectively $\MDolrig(X_S/S,  L_S,\leq r)$.
\item[(e)] For every point $y \in |\MdRrig(X_S/S,  L_S,\leq r)|$ there exists a family $(E_S,\nabla_S)$ as in (a) such that $y$ belongs to the set-theoretic image $[E_S,\nabla_S](|S|)$, and similarly for $\MDolrig(X_S/S,  L_S,\leq r)$
\end{enumerate}
\end{proposition}

\begin{proof}
Lemma \ref{lemma:arithmetic_models} implies the existence of an irreducible affine arithmetic scheme $\tilde{S}$ such that there is a model $(X_{\tilde{S}},L_{\tilde{S}})$ satisfying the conditions outlined there. We let $\tilde{R}$ be its ring of functions. By virtue of assumption it is embedded into $\Cb$.

Let $\mathcal{R}$ denote the set of finite type subrings $\tilde{R} \subset R \subset \Cb$ such that there is an arithmetic model $(X_R,L_R)$ satisfying the conditions of Lemma \ref{lemma:arithmetic_models}. 

Since the scheme $X$ is an inverse limit of the schemes $X_R$ where $R \in \mathcal{R}$, it follows from repeated application of  \cite[Th\'eor\`eme 8.5.2(i) \& Proposition 8.5.5]{EGA} that there exists $R \in \mathcal{R}$ such that conditions (a) and (b) are satisfied. Here we use that there are only finitely many rigid flat connections and Higgs bundles of rank $\leq r$ and determinant $L$.

{Furthermore, we have seen in \ref{nilpotent} that a rigid Higgs bundle $(V,\theta)$ on the complex scheme $X$ has nilpotent Higgs field.} In particular we have $\theta^{r}_{\Cb} = 0$. It follows from \cite[Th\'eor\`eme 8.5.2(i)]{EGA} that there exists $R \in \mathcal{R}$ such that $\theta^r_R = 0$. Therefore, conditions (a-c) hold.

Let $\{(E_i,\nabla_i)\}_{i=1,\dots,N}$ be a list of representatives of isomorphism classes of rigid flat connections on the complex projective variety $X$ of rank $\leq r$ and determinant $L$. By virtue of (a) there exist relative flat connections $\{(E_{i,S},\nabla_{i,S})\}_{i=1,\dots,N}$ over $X_S/S$, extending these representatives. By openness of $P$-stability we may assume that these families are $P$-stable over geometric points. We denote by 
$$s_i = [(E_{i,S},\nabla_{i,S})]\colon S \to \MdR(X_S/S,  L_S,\leq r)\text{ for }i=1,\dots,N$$
the associated $S$-points of the moduli space $\MdR(X_S/S,  L_S,\leq r)$. Let $\eta$ be the unique generic point of $S$. Since $k(\eta) \subset \Cb$, and by virtue of assumption we have $(E_i,\nabla_i) \in \MdRrig(X,  L,\leq r)$, it follows $s_i(\eta) \in \MdRrig(X_S/S,  L_S,\leq r)$ for $i=1,\dots,N$. The subset
$$U=\bigcap_{i=1}^N s_i^{-1}(|\MdRrig(X_S/S,  L_S,\leq r)|)$$
is therefore non-empty and open, since $\MdRrig(X_S/S,  L_S,\leq r) \subset \MdR(X_S/S,  L_S,\leq r)$ is open. Replacing $S$ by $U$ we have verified the first half of (d). The second half is treated the same way by replacing rigid flat connections by Higgs bundles.

It remains to prove (e). We assume that there is a model $(X_S,L_S)$ satisfying conditions (a-f). By construction, $\MdRrig(X_S/S,  L_S,\leq r) \to S$ is quasi-finite. Furthermore, using the notation introduced above, there are finitely many sections
$$s_i\colon S \to \MdRrig(X_S/S,  L_S,\leq r)\text{ for }i=1,\dots,N,$$
such that 
\begin{equation}\label{eqn:eta}
\bigcup_{i=1}^N \{s_i(\eta)\} = \MdRrig(X_S/S,  L_S,\leq r) \times_S \eta
\end{equation} 
We now apply Zariski's main theorem for quasi-finite maps \cite[Th\'eor\`eme 8.12.6]{EGA} and choose a factorization
$$\MdRrig(X_S/S,  L_S,\leq r) \hookrightarrow \widetilde{M} \xrightarrow{h} S,$$
where the first morphism is an open immersion and the second is finite. Furthermore, we may assume without loss of generality  that $\MdRrig(X_S/S,  L_S,\leq r)$ is dense in $\widetilde{M}$. We define 
$$Z= \overline{\left(|\widetilde{M}| \setminus \bigcup_{i=1}^N s_i(|S|)\right)}.$$
Since $\widetilde{M}$ is finite over $S$, the image $h(Z) \subset S$ is closed, and does not contain $\eta$ by virtue of \eqref{eqn:eta}. It follows that after replacing $S$ by $S \setminus h(Z)$ one obtains
$$|\MdRrig(X_S/S,  L_S,\leq r)| = \bigcup_{i=1}^N s_i(|S|).$$
This concludes the proof of (e) for flat connections. The case of Higgs bundles is dealt with mutatis mutandis.
\end{proof}


\subsection{Cartier transform and rigidity}\label{sub:cartiertransform}
 
Henceforth we let $S$ and $(X_S,L_S)$ be nice models as in Proposition \ref{prop:nice_models}. Smoothness of $S$ over $\Spec \Z$ implies for a closed point $s\in S$ with residue field $k(s)$ the existence of a lift $\Spec W_2(k(s)) \to S$. By base change one obtains a lift $X_{ W_2(k(s)) }$ of $X_s$ to $W_2(k(s))$. 

  \begin{lemma}\label{lemma:rigidmeansrigid}
Let $k$ be a perfect field of positive characteristic $p$, and let $Z/k$ be a smooth projective $k$-variety.  Then a stable Higgs bundle $(V,\theta)$ on $Z$ is rigid if and only if the stable Higgs bundle $w^{*}(V,\theta)$ is rigid, and a stable integrable connection $(E,\nabla)$ is rigid if and only if the stable integrable connection $w^{*}(E,\nabla)$ is rigid. 
\end{lemma}

\begin{proof}
{Recall that $(V,\theta)$ is not rigid if and only if there exists a geometrically irreducible $k$-scheme $C$ of finite type, with $\dim C > 0$ and a $C$-family of Higgs bundles $(V_C,\theta_C)$, $\theta_C: V_C\to V_C \otimes \Omega^1_{Z\times_k C/C}$, and two closed points $c_0,c_1 \in C(k')$ defined over a finite field extension $k'/k$, such that $(V_C,\theta_C)_{c_0}$ is isomorphic  to $(V, \theta)_{k'}$, and $(V_C,\theta_C)_{c_0}$ and $(V_C,\theta_C)_{c_1}$ are not isomorphic over the algebraic closure $\overline{k}$.}
We have a canonical isomorphism of schemes
$$(Z\times_k C)'=Z'\times_{k} C',$$
where the $k$ structure on the right is the one from the left in ${\rm Spec }(k)\xrightarrow{w} {\rm Spec }(k)$. And the functor $w_{Z\times_k C}$ induces an equivalence between $C$-families of stable Higgs bundles
on $Z$ and $C'$-families of stable Higgs bundles on $Z'$. This concludes the proof in the Higgs case. A similar strategy applies to the de Rham case.
\end{proof}
Subsequently, we simplify the notation of \eqref{3} and use the notation $\A'$ instead of $\A_{Z',r}$. The role of the variety $Z$ in \eqref{3} will be played by $X_s$ where $s\in S$ is a closed point and $X_S/S$ a nice model of $X$. 

Let $T$ be a $k(s)$-scheme. 
 We say that a $T$-point $a: T  \to  \A'$ is {\it OV-admissible}, if the spectral cover $X'_{s,a} \subset T^*X'_{s}\times_{k(s)} T $  (see  Remark \ref{rmk:spec})
 factors through the $(p-1)$-st order infinitesimal neighbourhood of the zero section $X_s' \hookrightarrow T^*X_s'$.  We assume that the characteristic $p$ of $k(s)$ is $\ge r+2$. {This assumption guarantees that the level of nilpotency of the Higgs field lies in the range that is needed in order to apply the Ogus--Vologodsky correspondence recalled in Subsection \ref{sub:OV}. }
\begin{proposition}\label{Crigid}
Let $(X_S,L_S)$ be a nice model of $X$ as in Proposition \ref{prop:nice_models}. There exists a positive integer $D$, depending only on $X_S/S$ such that for any closed point $s \in S$ with $\mathrm{char}\; k(s) > D$, and any rigid stable Higgs bundle $(V_s,\theta_s)$, the { inverse} Cartier transform $C^{-1}(V'_s,\theta'_s)$ is a stable rigid integrable connection. 
\end{proposition} 

\begin{proof}
Stability is proven by following the argument of \cite[Cor. 5.10]{Lan14}, which shows semistability.  The main point is the rigidity assertion, which we prove now. 
Let $T$ be a $k(s)$-scheme. By  the BNR correspondence (Theorem \ref{thm:BNR}) we may describe a $T$-family of flat connections $(E_T,\nabla_T)$ in terms of a pair $(a\colon T \to \A',M)$ where $M$ is a $\mathcal{D}_{X'_s}|_{X'_{s,a}}$-module on the spectral cover $X'_{s,a} \hookrightarrow T^*X_s' \times_{k(s)} T$.

If $a\in \A'(T)$ is OV-admissibble,  we may apply Ogus--Vologodsky's result that $\mathcal{D}$ splits on the $(p-1)$-st order neighbourhood of $X'_s\hookrightarrow T^*X_s'$ (see \cite[Cor. 2.9]{OV07}), and hence obtain for every OV-admissible $a \in \A'(T)$ that the stack of stable flat connections $(E,\nabla)$ on $X_s$ with $\chi_{dR}( (E,\nabla)) = a$  (defined in \eqref{2}) is equivalent to the stack of stable Higgs bundles $(V',\theta')$ on $X'_s$ with $\chi( (V', \theta')) = a$. We denote this equivalence by $C_a$, respectively $C^{-1}_a$.

There exists a polynomial function $R(r,m)$ (linear in $m$ and quadratic in $r$), with the following property: let $m$ be a positive integer, and let $\A^{\prime(m)}$ be the $m$-th order neighbourhood of $0 \in \A'$. If {we have} $p-1 > {R(r,m)}$, then every scheme theoretic point $a: T \to  \A^{\prime (m)}$ is OV-admissible.

Let $D'$ be a positive integer such that $D'$ is bigger than the degrees of the finite morphisms $\MDolrig(X_S/S, L_S) \to S$ and $\MdRrig(X_S/S, L_S) \to S$. Let $k(s) \to B \to k(s)$ be an augmented Artinian local algebra such that we have a $B$-deformation $(V_B,\theta_B)$. The characteristic polynomial of $\theta_B$ defines a point $\chi: \Spec B\to \A'$. Then $\chi$ factors through the $(D'-1)$-st order infinitesimal neighbourhood $\A^{\prime (D'-1)}$. To see this we observe that we have a commutative diagram
\[
\xymatrix{
\Spec B \ar[r]^-{\phi} \ar[rd]_{\chi} &  \MDolrig(X'_s/k(s), L'_s) \ar[d] \\
& \A'.
}
\]
The morphism $\phi$ factors through the connected component of the moduli space corresponding to the isolated point $[(V_s,\theta_s)]$. This shows that $\chi$ factors through the $(D'-1)$-st order infinitesimal neighbourhood.

Let $m$ be a positive integer such that  $m > D'$, and assume $p-1 > {R(r,m)}$, such that every scheme theoretic point $a: T\to   \A^{\prime (m) }$ is OV-admissible.

We assume by contradiction that $(E,\nabla) = C^{-1}(V'_s,\theta'_s)$ is not rigid as a local system. This implies that there exists a deformation $(E_T,\nabla_T)$, parametrized by an augmented $k(s)$-scheme $\Spec k(s) \to T \to \Spec k(s)$, so that the corresponding Hitchin invariant $\chi_{dR}( (E_T, \nabla_T))$ does not factor through $\A^{\prime (m-1)}$. 

 We denote by $T_t^{(m-1)}$ be the $(m-1)$-st order neighbourhood of $t={\rm Im}(\Spec k(s))$ in $T$. By construction the family $(E_{T_t^{(m-1)}},\nabla_{T^{(m-1)}_t})$ has the property that $\chi_{dR}((E_{T_t^{(m-1)}},\nabla_{T^{(m-1)}_t}))$ is a morphism $T^{(m-1)}_t \to \A'$, and therefore it factors through {$\A^{\prime (m-1)}$}, but not through $\A^{(k-1)}$ for $k < m$. 
 
 For every $p - 1 > {R(r,m)}$,  we can apply the equivalence of categories $C_{a}$ to construct a $T^{(m-1)}_t$-deformation $(V'_{T^{(m-1)}_t},\theta'_{T^{(m-1)}_t})$ of $(V'_s,\theta'_s)$ such that $\chi(\theta'_{T^{(m-1)}_t})$ does not factor through $\A^{\prime (D'-1)}$. This {is a contradiction}.
\end{proof}

{
We introduce new notation before turning to the consequences of the result proved above.
\begin{definition}
Let $Z/k$ be a smooth projective variety, $L$ be a line bundle of finite order, invertible in $k$.  Let $n_{dR}(Z, L, r)$ denote the number of isomorphism classes of stable rigid flat connections of rank $r$ with determinant isomorphic to $L$  on $Z$. Let $n_{Dol}(Z, L, r)$ the number of isomorphism classes of stable rigid Higgs bundles of rank $r$  with determinant isomorphic to $L$ on $Z$.
\end{definition}
If $L$ is a line bundle on $Z$ and $Z'$ is the Frobenius twist of $Z$, we denote by $L'$ the Frobenius twist of $L$, that is the pullback of $L$ under the map $Z'\to Z$.
\begin{corollary}\label{cor:ndR}
Let $D$ be the positive integer of Proposition \ref{Crigid}. Let $s \in S$ be a closed point such that $\mathrm{char}(s) > D$. If $n_{dR}(X_s, L_s, r) = n_{Dol}(X_s', L'_s,r)$ then every stable rigid flat connection $(E,\nabla)$ of rank $r$ {and determinant line $L$} on $X_s$ has nilpotent $p$-curvature.
\end{corollary}

\begin{proof} 
Let $n_{dR}^{nilp}(X_s, L_s, r)$ be the number of isomorphism classes of stable rigid flat connections $(E,\nabla)$ of rank $r$ with determinant $L_s$ on $X_s$ which have nilpotent $p$-curvature. By definition we have $n_{dR}^{nilp}(X_s, L_s, r) \leq n_{dR}(X_s,L_s, r)$ and equality holds if and only if every stable rigid $(E,\nabla)$ of rank $r$  with determinant $L_s$ has nilpotent $p$-curvature.
By Proposition \ref{Crigid},  for $\mathrm{char}(s) > D$,  the functor $C^{-1}$ sends a rigid Higgs bundle to a rigid flat connection. By definition of $C^{-1}$ the latter has nilpotent $p$-curvature. We therefore conclude that $n_{dR}^{nilp}(X_s,L_s, r) \geq n_{dR}(X_s,L_s, r)$ which shows $n_{dR}^{nilp}(X_s,L_s, r) = n_{dR}(X_s,L_s, r)$.
\end{proof}

}

\subsection{Proof of Theorem \ref{thm:nilp}} \label{proof:nilp}
{
As in Corollary \ref{cor:ndR} we denote by $n_{dR}(Z, L, r)$ the number of stable rigid rank $r$ flat connections of determinant $L$ on $Z/k$. We use the notation $n_{Dol}(Z,L, r)$ to refer to the number of stable rigid rank $r$ Higgs bundles on $Z/k$ of determinant $L$. 

\begin{proof}
Recall that we have a smooth complex projective variety $X/\mathbb{C}$ and a torsion line bundle $L$ as well as an appropriately chosen model $(X_S/S, L_S)$ {as in Proposition \ref{prop:nice_models}}. The numbers $n_{dR}(X,L, r)$ and $n_{Dol}(X,L, r)$ are equal by virtue of the Simpson correspondence (see \cite[Section 4]{Sim92}) between stable Higgs bundles and irreducible flat connections on $X/\mathbb{C}$. Furthermore, we know from Simpson's observation (see Lemma \ref{nilpotent}) that a rigid Higgs bundle has nilpotent Higgs field.

 For every closed point $s \in S$ one has  $n_{dR}(X,L, r) = n_{dR}(X_s,L_s, r)$ and $n_{Dol}(X, L, r) = n_{Dol}(X'_s,L'_s, r)$ (using Lemma \ref{lemma:rigidmeansrigid}). 
In particular, {we have} $n_{dR}(X_s, L_s, r) = n_{Dol}(X_s', L'_s, r)$ and therefore Corollary \ref{cor:ndR} implies for $\mathrm{char}(s) > D$ that every stable rigid flat connection $(E,\nabla)$ on $X_s$ has nilpotent $p$-curvature.
\end{proof} 
}

\section{Frobenius structure} \label{sec:Fstr}
This section is devoted to proving Theorem \ref{thm:Fstr}. { Our proof is based on the theory of Higgs-de Rham flows as developed by Lan--Sheng--Zuo in  \cite{LSZ13}. We begin by recalling their results.
}
\subsection{Recollection on Lan--Sheng--Zuo's Higgs-de Rham flows}\label{LSZ}
{

As before we denote by $k$ a perfect field of positive characteristic $p$ and by $Z/k$ a smooth $k$-variety {that admits a lift to $W_2(k)$}. {We denote by $w=w_Z\colon Z' \to Z$ the isomorphism of schemes induced by the arithmetic Frobenius $k \to k$ by base change}. We have seen Ogus--Vologodsky's Cartier transform $C\colon \Higgs(Z)_{p-1} \to \MIC(Z)_{p-1}$ in Theorem \ref{thm:OV}. 

\begin{definition}\label{defi:Cartier1}
\begin{itemize}
\item[(a)] Let $\iota\colon \Higgs(Z) \to \Higgs(Z)$ be the autoequivalence given by $(E,\theta) \mapsto (E,-\theta)$.
\item[(b)] We denote by $C_1\colon \Higgs(Z)_{p-1} \to \MIC(Z)_{p-1}$ the composition $C^{-1} \circ (w^{-1})^* \circ \iota$.
\item[(c)] A \emph{Higgs-de Rham fixed point} is a quadruple $(H,\nabla,F,\phi)$, where $(H,\nabla)$ is a vector bundle with a flat connection of level $\leq p-1$, $F$ is a descending filtration on $H$ satisfying the Griffiths transversality condition, and $\phi\colon C_1^{-1}(\mathrm{gr}^F(H),\mathrm{gr}^F(\nabla)) \simeq (H,\nabla)$ is an isomorphism in $\MIC(Z)_{p-1}$. 
\end{itemize}
\end{definition}

Subsection 4.6 of \cite{OV07} shows that the category of Higgs-de Rham fixed points is equivalent to the category of $p$-torsion Fontaine--Lafaille modules as defined in \cite{FL82}. If $Z$ admits a lift to $W(k)$, then the category of Fontaine--Lafaille modules admits a fully faithful functor to the category of \'etale local systems of $\mathbb{F}_p$-vector spaces on $Z_K$, where $K = \mathrm{Frac}(W)$ (see \cite[Theorem 3.3]{FL82} for a special case and Faltings \cite[Theorem 2.6*]{Fal88}).

Lan--Sheng--Zuo generalize this by replacing fixed points by periodic orbits with respect to the so-called Higgs-de Rham flow. This variant gives rise to \'etale local systems of $\mathbb{F}_q$-vector spaces instead. We recall their definition below.  

\begin{definition}[Lan--Sheng--Zuo]\label{defi:HdR_flow}
An $f$-periodic flat connection on $Z$ is a tuple 
$$(E_0,\nabla_0,F_0,\phi_0,E_1,\nabla_1,F_1,\dots,E_{f-1},\nabla_{f-1},F_{f-1},\phi_{f-1}),$$ 
where for all {$i \in \mathbb{Z}/f\mathbb{Z}$} we have that $(E_i,\nabla_i,F_i)$ is a nilpotent flat connection on $Z$ of level $\leq p-1$ with a Griffiths-transversal filtration $F_i$, and $\phi_i\colon C_1^{-1}(gr^F E_i,gr(\nabla_i)) \simeq (E_{i+1},\nabla_{i+1})$.
\end{definition}

The direct sum $\bigoplus_{i=0}^{f-1}(E_i,\nabla_i,F_i)$ is by definition a Higgs-de Rham fixed point. Cyclic permutation of the summands induces an automorphism of order $f$. Using the aforementioned connection between Higgs-de Rham fixed points and $p$-torsion Fontaine--Lafaille modules one obtains the following result (see \cite[Corollary 3.10]{LSZ13}). {In \emph{loc. cit.} the authors assume that $k$ be algebraically closed. However, this assumption is not needed, see Lemma 1.2 in \cite{syz} for a more general version.}

\begin{proposition}[Lan--Sheng--Zuo, {Sun--Yang--Zuo}]\label{prop:lsz}
Assume that $Z$ can be lifted to a smooth $W$-scheme $Z_W/W$, and that $k \supset \mathbb{F}_{p^f}$. The category of $f$-periodic flat connections on $Z$ admits a fully faithful functor to the category of \'etale local systems of $\mathbb{F}_{p^f}$-vector spaces on $Z_K$, where $K=\mathrm{Frac}(W)$.
\end{proposition}

So far we have only been considering flat connections $(E,\nabla)$ on the special fibre $Z/k$, and \'etale local systems of vector spaces over finite fields. The theory of \cite{LSZ13} also works in a mixed characteristic setting. Let $W=W(k)$ and $Z_W / W$ be a smooth $W$-scheme. We denote by $Z_n$ the fibre product $Z_W \times_{\Spec W} \Spec W_n$.

In the following definition it is necessary to work with nilpotent connections of level $\leq p-2$ rather than $p-1$.

{

\begin{definition}\label{defi:h}
For every natural number $n$ we define the category $\mathcal{H}(Z_n/W_n)$ to be the category of tuples $(V,\theta, \bar{E},\bar{\nabla},\bar{F},\phi)$, where $(V,\theta)$ is a graded Higgs bundle on $Z_n$, $(\bar{E},\bar{\nabla},\bar{F})$ is a flat connection on $Z_{n-1}$ with a Griffiths-transversal filtration $\bar{F}$, and $\phi\colon gr_F(\bar{E},\nabla) \simeq (V,\theta) \times_{W_n} W_{n-1}$ is an isomorphism of graded Higgs bundles. Furthermore, we assume that the $p$-curvature of $(\bar{E},\bar{\nabla}) \times_{W_{n-1}} k$ is nilpotent of level $\leq p-2$. 
\end{definition}

Similarly we denote by $\MIC(Z_n/W_n)$ the category of quasi-coherent sheaves with $W_n$-linear flat connections on $Z_n$. We have the following result \cite[Theorem 4.1]{LSZ13} (for $k$ algebraically closed and \cite[1.2.1]{syz} for the case of finite fields).} {Closely related results were obtained by Xu in \cite{Xu19}.}

\begin{theorem}[Lan--Sheng--Zuo, {Sun--Yang--Zuo}]\label{thm:cn}
For a positive integer $n$ there exists a functor 
$$C_n^{-1}\colon {\mathcal{H}}(Z_n/W_n)\to \MIC(Z_n/W_n)$$ 
which extends the one of Definition \ref{defi:Cartier1}(b) over the special fibre.
\end{theorem}

{For a $W_n$-linear flat connection with a Griffiths-transversal filtration $(E,\nabla,F)$ we write $\overline{gr}(E,\nabla,F)$ to denote the tuple $(gr_F(E),gr_F(\nabla),(E,\nabla,F)_{W_{n-1}},\id)$.
This allows one to extend the notion of periodic flat connections.}

\begin{definition}[Lan--Sheng--Zuo]\label{defi:f-periodic}
{We assume $k \supset \mathbb{F}_{p^f}$}. An $f$-periodic flat connection on $Z_W/W$ is a tuple 
$$(E_0,\nabla_0,F_0,\phi_0,E_1,\nabla_1,F_1,\dots,E_{f-1},\nabla_{f-1},F_{f-1},\phi_{f-1}),$$ 
where for all $i$ we have that $(E_i,\nabla_i,F_i)$ is a flat connection on $Z_W$ ({nilpotent of level $\leq p-2$} on the special fibre) with a Griffiths-transversal filtration $F_i$, {such that for all integers $n$ we have that $\overline{gr}_F(E_i,\nabla_i)$ belongs to $\mathcal{H}_{Z_n/W_n}$ and $\phi_i\colon C_1^{-1}(gr^F E_i,gr(\nabla_i)) \simeq (E_{i+1},\nabla_{i+1})$}.
\end{definition}

By taking an inverse limit (of $p^n$-torsion Fontaine--Lafaille modules) of Proposition 5.4 \cite{LSZ13} we obtain a mixed characteristic version of Proposition \ref{prop:lsz}. {We refer the reader to \cite[Theorem 1.4]{LSZ13} (for $k$ algebraically closed) and \cite[Lemma 1.2]{syz} for $k$ being finite.} 

\begin{theorem}[Lan--Sheng--Zuo, {Sun--Yang--Zuo}]\label{thm:lsz}
{Assume $k \supset \mathbb{F}_{p^f}$.}
\begin{itemize}
\item[(a)] There exists an equivalence between the category of $1$-periodic flat connections on $Z_W/W$ and the category of torsion-free Fontaine--Lafaille modules on $Z_W$.
\item[(b)] There exists a fully faithful functor from the category of $f$-periodic flat connections on $Z_W/W$ to the category of crystalline \'etale local systems of free $W(\mathbb{F}_{p^f})$-modules on $Z_K$.
\end{itemize}
\end{theorem}

Torsion-free Fontaine--Lafaille modules are also known as strongly $p$-divisible lattices of an $F$-isocrystal. In light of this, we obtain a criterion for a $W$-family of flat connections $(E_W,\nabla_W)$ to give rise to an $F$-isocrystal.

\begin{corollary}\label{cor:f-periodic}
{Assume $k \supset \mathbb{F}_{p^f}$} and let $(E_W,\nabla_W)$ be a flat connection on $Z_W/W$ which is $f$-periodic. Then the formal flat connection $(\hat E, \hat \nabla)$ obtained by pullback to the formal completion of $X_W$ is an isocrystal with Frobenius structure.
\end{corollary}
}

\subsection{Higgs-de Rham flows for rigid flat connections}

Let $T$ be a smooth scheme over $\C$ and $\lambda\colon T \to \Ab^1_{\Cb}$ a regular function. A \emph{$\lambda$-connection} on a vector bundle $N$ is a $\Cb$-linear map of sheaves
$$D \colon N \to N \otimes \Omega_T^1,$$
such that for every open subset $U \subset T$ and sections $s \in N(U)$ and $f \in \Oo_T(U)$ we have 
$$D(fs) = fD(s) + \lambda s \otimes df.$$
We say that $D$ is integrable (or flat), if it satisfies $D^2 = 0$.

There are two special cases which are of particular interest to us: for $\lambda= 0$ a flat $\lambda$-connection amounts to a Higgs field, and for $\lambda = 1$, a flat $\lambda$-connection is a flat connection in the classical sense.

For a smooth and projective scheme $X/\Cb$ and a torsion line bundle $L\in \Pic(X)$ we denote by $\MHod(X/\Cb,L,r)$ the moduli space of $P$-stable pairs $(N,D)$, where $N$ is a rank $r$ vector bundle on $X$ and $D$ a $\lambda$-connection for $\lambda \in \Cb$. We refer the reader to \cite[p. 87]{Sim94} for more details. By definition, there is a morphism
\begin{equation}\label{eqn:A1}
\MHod(X/\Cb,L,r) \to \Ab^1_{\Cb},
\end{equation}
such that we have isomorphisms
$$\MHod(X/\Cb,L,r) \times_{\Ab^1_{\Cb}} \{0\} \simeq \MDol(X/\Cb,L,r)\text{ and }\MHod(X/\Cb,L,r) \times_{\Ab^1_{\Cb}} \{1\} \simeq \MdR(X/\Cb,L,r).$$

We define $\MHodrig(X/\C,L,r)  \subset \MHod(X/\C, L,r)$ to be the maximal open subset where \eqref{eqn:A1} is quasi-finite (see Definition \ref{defi:rig}).

As before, we use the notation $[(N,D)]$ to denote the point of the moduli space $\MHod(X/\C, L,r)$ induced by a $\lambda$-connection $(N,D)$ on $X$.
\begin{lemma} \label{lem:simpsonsplit}
 The morphism  $\MHodrig(X/\C, L,r) \to \A^1$ is finite, flat, and splits $\G_m$-equivariantly as 
 $$\MHodrig(X/\C, L,r) \cong \MDolrig(X/\C, L,r) \times_{\C} \A^1  \xrightarrow{\rm ( (V, \theta), \lambda)  \mapsto \lambda} \A^1.$$
 For $[(N,D)]$ a section, then $[(N,D)]_{1}\in \MdR(X/\C, L,r)$ is the moduli point of a complex variation of Hodge structure, with $F$-filtration $F^i\subset F^{i-1} \subset \ldots F^0=N$ with Griffiths transversality $\nabla\colon F^i\to  \Omega^1_X\otimes_{\sO_X} F^{i-1}$, and  $[(N,D)]_{0}\in \MDol(X/\C,L,r)$ is the moduli point of the associated Higgs bundle 
 $$\big(V=gr_F E, gr_F \nabla\colon \bigoplus_i (gr_F^i E\to \Omega^{1}_X\otimes_{\sO_X} gr_F^{i-1} E)\big).$$ 
\end{lemma}
\begin{proof}
By construction $\MHod(X/\C,L,r)\times_{\A^1} \G_m \to \G_m$ splits as $$\MHod(X/\C,L,r)\times_{\A^1} \G_m  \xleftarrow{ \cong  \  ((E,\nabla), \lambda) \mapsto    (E,\lambda \nabla) } \MdR(X/\C,L,r)\times_k  \G_m \xrightarrow{\rm  ((E,\nabla), \lambda) \mapsto \lambda  } \G_m,$$ where 
$\MdR(X/\C,L,r)$ is the fibre at $\lambda=1$.   
On the other hand, by \cite[Theorem 9.1]{Sim97}, at a complex point $x\in \MDol(X/\C,L,r)$, the fibre at $\lambda=0$,  $\MHod(X/\C,L,r)$ is \'etale locally isomorphic to the product  of
$\MDol(X/\C,L,r)$ with $\A^1$.  This finishes the proof of the first part. As for the second part, this is an application of  \cite[Lemma 7.2]{Sim97}.
\end{proof}

\medskip

{Consider an arithmetic scheme $S$ and a smooth model $(X_S, L_S)$ as in Lemma \ref{lemma:arithmetic_models}. For} every $\lambda \colon S \to \mathbb{A}^1,$ Langer's construction \cite[Theorem 1.1]{Lan14} yields a coarse moduli space of semistable $\lambda$-connections $\mathcal{M}_{\lambda}(X_S/S)$ defined over $S$. In particular we can apply this to the case 
$$\lambda = p_{\mathbb{A}^1}\colon S \times \mathbb{A}^1 \to \mathbb{A}^1$$
and obtain an $S$-model of Simpson's Hodge moduli space $\MHod(X_S/S, L_S) \to S \times \Ab^1$. 
In the following proposition we denote by $d$ the order of the torsion line bundle $L$ on $X$.

\begin{proposition}[Nice models 2]\label{prop:nice_models2}
For every positive integer $r$ there exists an affine arithmetic scheme $S$ and a model $(X_S,L_S)$ of $(X,L)$, such that the following properties are satisfied:
\begin{enumerate}
\item[(a)] all properties of Proposition \ref{prop:nice_models},

\item[(b)] there are finitely many $\lambda$-connections $(N_S^i,D_S^i)_{i=1,\dots,M}$ on $X_S \times_{S}\Ab^1_S$ with respect to $\lambda = \mathsf{pr}_2$, and furthermore we assume that $(N_S^i,D_S^i)$ is geometrically $P$-stable,

\item[(c)] the $\lambda$-connections of (b) give rise to a bijection
$$\bigsqcup_{i=1}^M [(N_S^i,D^i_S)](|S|) = \bigsqcup_{a=0}^{d-1}|\MHodrig(X,L,\leq r)|.$$
\end{enumerate}
\end{proposition}

\begin{proof}
This can be shown using the same techniques as for the proof of Proposition \ref{prop:nice_models}.
\end{proof}

  \medskip
  
  Henceforth we choose an arithmetic model $(X_S,L_S)$ as in Proposition \ref{prop:nice_models2}. For every closed point $s \in S$ we apply the theory of Higgs-de Rham flows as recalled in Subsection \ref{LSZ}. 

  We fix a closed point $s$ of the scheme $S$, and also choose a lift to a $W_2(k(s))$-point of $S$ (using that $S$ is smooth over $\Spec \Zb$).  We denote by $n_{L}$ the number of rank $r$ rigid flat connections on $X$ with determinant isomorphic to $L^a$ for $a=0,\dots,d-1$, and choose a bijection $\{1,\dots,n_L\} \simeq \bigsqcup_{a=0}^{d}\MdRrig(X/\Cb,L^a,R)(\Cb)$ and define a map
   \ga{sigma}{\sigma: \{1,\ldots, n_L\}\to \{1,\ldots, n_L\}}
   as follows:    
  For $i\in \{1,\ldots, n_L\}$, 
   we  set {$(N^i_S, D^i_S)_{0_S}\times_S 0= (V_s^i,\theta_s^i) \in \sqcup_{a=0}^{d-1} \MDolrig(X/\C, L,r)$}.  One first defines 
   \ga{prime}{(V_s^{'i},\theta_s^{'i})= w^* (V_s^i,\theta_s^i),}
  which by  Lemma \ref{lemma:rigidmeansrigid}  is rigid stable, then one defines
  $C^{-1}(V^{'i},\theta^{'i})$
  which by Proposition \ref{Crigid} is a stable rigid integrable connection. Thus, there is a uniquely defined 
  $\sigma(i)\in \{1,\ldots, n_L\}$ such that 
  \ga{dRH}{C^{-1}(V_s^{'i},\theta_s^{'i})=(E_s^{\sigma(i)}, \nabla_s^{\sigma(i)}).}
  
  \begin{lemma} \label{lem:bij}
  The map $\sigma$ is a bijection.  
  \end{lemma}
  \begin{proof}
  Clearly,  we can reverse the argument: starting with $(E_s^j,\nabla_s^j)$, then $C(E_s^j,\nabla_s^j)$ again is stable and rigid by Lemma \ref{lemma:rigidmeansrigid}, and thus $w^{-1*}C(E_s^j,\nabla_s^j)$ as well. 
  \end{proof}
  Fixing $i\in \N, 1\le i\le n_L$, we define a Higgs-de Rham flow (for the definition, see \cite[Definition 1.1]{LSZ13}) using $\sigma$ as follows:
  \ga{flow}{ \xymatrix{ &  (E_s^{\sigma(i)}, \nabla_s^{\sigma(i)}) \ar[dr] & & (E_s^{\sigma^2(i)}, \nabla_s^{\sigma^2(i)}) \ar[dr] &  \\
(V_s^i, \theta_s^i) \ar[ur] &  & (V_s^{\sigma(i)}, \theta_s^{\sigma(i)}) \ar[ur] & & \ldots
  }}
  The downwards arrows are obtained by taking the graded associated to the restriction to $X_s$ of the  Hodge filtration on the rigid connections. 
  
    \begin{lemma} \label{lem:per}
    \begin{itemize}
    \item[1)]
  The Higgs-de Rham flow \eqref{flow} is periodic of period $f_i$ which is the order of the $\sigma$-orbit of $i$.
  \item[2)] It does not depend on the choice of the $W_2(k(s))$-point of $S$ chosen.
  \end{itemize}
  
  \end{lemma}
  \begin{proof}
  The point 1) holds by definition, since the map $\sigma$ is shown to be a bijection in Lemma \ref{lem:bij}. The second assertion can be seen to be true as follows. We fix a $W_2(k(s))$-lift of $X_s$. Its Kodaira--Spencer class   endows  the set of equivalence classes of $W_2(k(s))$-lifts of $X_s$ with $k(s)$-points of an affine space $A$. Combining Proposition \ref{Crigid} with the operator $C^{-1}$,  one obtains an $A$-family of isomorphisms of moduli spaces
$$\MdRrig(X_s/s, L_s^p,r)  \times_{k(s)} A\simeq \MDolrig(X'_s/s, L'_s,r) \times_{k(s)} A.$$
Since the moduli spaces are $0$-dimensional, the resulting bijection of closed points is independent of the chosen $W_2$-lift.
\end{proof}

We fix now a $W(k(s))$-point of $S$, yielding $X_{W(k(s))}$. For an irreducible rigid $(E_S,\nabla_S)$ we show that $(E_{W(k(s))},\nabla_{W(k(s))})$ is $f$-periodic (see Definition \ref{defi:f-periodic}) and therefore we conclude using Corollary \ref{cor:f-periodic} that the isocrystal $(E,\nabla)$ has a Frobenius structure (after) an unramified field extension of $K(s) = \mathsf{Frac}\; W(k(s))$.
%
  \begin{proof}[Proof of Theorem \ref{thm:Fstr}]   
  
For the duration of this proof we introduce the shorthand $\MdRrig = \sqcup_{a=0}^{d-1}\MdRrig(X_s/s, L^a_s,r)$. We will also use $\MDolrig$ and $\MHodrig$ to denote disjoint unions as the one above.
  Recall that we have chosen a finite type scheme $S$ over $\mathbb{Z}$ such that every rigid connection has an $S$-model. { We choose $S$ as in Proposition \ref{prop:nice_models2}. In particular,} we may assume that for every $s \in S$, every rigid connection $(E_s,\nabla_s)$ over $X_s$ is the restriction of a unique $S$-model of a rigid connection. In particular we have that $\MdRrig(X_S/S, L_S)^{\rm red} \to S$ is an isomorphism of schemes on every connected component. Using that $W(k(s))$ is reduced, we obtain:
  
  \begin{claim}\label{claim:bijection} For every $s \in S$ with residue field $k(s)$ 
the closed embedding $s\to {\Spec}W(k(s))$ induces a bijection
     $$\MdRrig(k(s)) = \MdRrig(W(k(s))).$$ 
     \end{claim}
Furthermore, recall from Lemma \ref{lem:simpsonsplit} that we have a $\G_m$-equivariant isomorphism 
$$\MHodrig \cong \MDolrig \times \Ab^1.$$    
For a positive integer $i$ we denote by $W_i = W_i(k(s))$ the $i$-truncated Witt ring. The isomorphism above implies the assertion: for every $y \in \MdRrig(W_i(k(s)))$ there exists a unique $\G_m$-equivariant section $\mathbb{A}^1_{W_i(k(s))} \to \MHodrig$, sending $1 \in \mathbb{A}^1_{W(k(s))}$ to $y$. Similarly, every $\G_m$-fixpoint $z \in \MDolrig(W_i)^{\G_m}$ extends to a unique $\G_m$-equivariant section $\mathbb{A}^1_{W_i(k(s))} \to \MHodrig$ sending $0$ to $z$.
This yields:
\begin{claim}\label{claim:F_unique}
A rigid $W_i(k(s))$-family of stable flat connections $(E,\nabla) \in \MdRrig$ has a unique Griffiths-transversal filtration $F$ (up to shifting the filtration). And, a rigid $W_i(k(s))$-family of stable Higgs bundles $(V,\theta)$ is isomorphic to the associated graded of a Hodge bundle $(E,\nabla,F)$, which is unique up to isomorphism.
\end{claim} 

For a positive integer $i>1$ we let $\mathcal{H}_i$ be the set of isomorphism classes of tuples $(V,\theta,\bar{E},\bar{\nabla},\bar{F},\bar{\phi})$ over $X_i/W_i$ as in Definition \ref{defi:h}, with the additional assumption that the underlying (ungraded) Higgs bundle $(V,\theta)$ is rigid and stable, that is, represents a $W_i$-point of $\MDolrig$.

\begin{claim}\label{claim:bijection2}
The forgetful map $\mathcal{H}_i \to \MDolrig(X_S/S)(W_i)$ is a bijection.
\end{claim}

\begin{proof}
Given $(\bar{V},\bar{\theta})=(V,\theta) \times_{W_i} W_{i-1}$, we apply Claim \ref{claim:F_unique} to deduce that there is a (up to isomorphism) unique $(\bar{E},\bar{\nabla},\bar{F})$ such that we have an isomorphism $\bar{\phi}$ between $(\bar{E},\bar{\theta})$ and the associated graded of $(\bar{E},\bar{\nabla},\bar{F})$.
\end{proof}   

We therefore see that Lan--Sheng--Zuo's functor $C_i^{-1}$ (see Theorem \ref{thm:cn} above) gives rise to a map 
$$C^{-1}_i\colon \MDolrig(W_i) \to \MdRrig(W_i).$$

Furthermore, Claim \ref{claim:bijection} yields a map $\overline{\gr}\colon \MdRrig(W_i) \to \MDolrig(W_i)$ which corresponds to the construction $\overline{\gr}(E,\nabla,F)$ recalled below Theorem \ref{thm:cn}. This is simply the case, because we have a unique Griffiths-transversal filtration $F$ for every rigid $W_i$-family $(E,\nabla)$.

This shows that the Higgs-de Rham flow for rigid flat connections actually corresponds to a self-map of sets
$$\mathrm{fl}_i\colon \MdRrig(W_i) \to \MdRrig(W_i).$$

We have 
$$\MdRrig(X_S/S)(W(k(s))) = \varprojlim_i \MdRrig(X_S/S)(W(k(s))/\mathfrak{m}^i).$$
We introduce the notation $M_i = \MdRrig(X_S/S)(W(k(s))/\mathfrak{m}^i)$, $M_W = \MdRrig(X_S/S)(W(k(s)))$ and 
$M_{i,0} = \bigcap_{j > i}\mathrm{im}(M_j \to M_i)$.

An elementary argument for inverse limits shows that 
$$M_W = \varprojlim_i M_{i,0}.$$

\begin{claim}
The map of sets $M_{i+1,0} \to M_{i,0}$ is a bijection for all $i > 0$. For all $i >0$, the subset $M_{i,0}$ is preserved by the self-map ${\mathrm{fl}_i}\colon M_i \to M_i$.
\end{claim}

\begin{proof}
We have a commutative diagram
\[
\xymatrix{
M_W \ar@{->>}[r] \ar[rd]_{\simeq} & M_{i,0} \ar[d] \\
& M_1. 
}
\]
Recall that $M_1 = \MdRrig(X_S/S)(k) = M_W$ (see Claim \ref{claim:bijection}). Since the map $M_W \to M_{i,0}$ is surjective by construction, and injective by commutativity of the diagram, we see that $M_W \to M_{i,0}$ is a bijection for all $i > 0$. The commutative diagram
\[
\xymatrix{
M_W \ar[r]^{\simeq} \ar[rd]_{\simeq} & M_{i+1,0} \ar[d] \\
& M_{i,0}
}
\]
shows that $M_{i+1,0} \to M_{i,0}$ is bijective.

We turn to the proof of the second assertion: the inductive nature of $\mathrm{fl}_i$ reveals that 
\[
\xymatrix{
M_{i+1} \ar[r]^{\mathrm{fl}_{i+1}} \ar[d] & M_{i+1} \ar[d] \\
M_i \ar[r]^{\mathrm{fl}_i} \ar[r] & M_i
}
\]
commutes. This shows that for every $j > i$, the image of $M_j \to M_i$ is preserved by $\mathrm{fl}_i$.
Since $M_{i,0} = \bigcap_{j > i} \mathrm{im}(M_j \to M_i)$ we see that $M_{i,0}$ is preserved by $\mathrm{fl}_i$.
\end{proof}

This shows that length $f$ orbits under $\mathrm{fl}_i$ in $M_{i,0}$ are in bijection with the $f$-periodic Higgs-de Rham flows over $k(s)$ constructed in Lemma \ref{lem:per}. Furthermore, since $M_W \to M_{i,0}$ is a bijection, we deduce that every $W$-family of stable rigid flat connections $(E,\nabla)$ of rank $r$ gives rise to a periodic Higgs-de Rham flow. Since there are only finitely many rigid flat connections of rank $r$ there exists a positive integer $f_r$ such that the period length $f$ of every rigid rank $r$ connection divides $f_r$. An unramified field extension of $K(s) = \mathsf{Frac} \; W(k(s))$ of degree $f_r$ satisfies the conclusion of Theorem \ref{thm:Fstr} by virtue of Corollary \ref{cor:f-periodic}.
\end{proof}


  \begin{corollary} \label{cor:mot}
  Let $X$ be a smooth projective variety over $\C$ and let $C\hookrightarrow X$ be a complete intersection of ample divisors of dimension one. We denote by $X_S$ an arithmetic model as in Proposition \ref{prop:nice_models2} such that also $C_S$ has a model over $S$. Then there is a non-empty open subscheme $S'\hookrightarrow S$ such that for all closed points $s\in S$,  all $W(k(s))$-points of $S'$, 
  there is a  projective morphism $f_s: Y_s\to C^0_s$ on a dense open $C^0_s\hookrightarrow C_s$ such that the $F$-overconvergent isocrystal $(E_{K(s)}, \nabla_{K(s)})|_{C_{K(s)}}$  is, over a finite extension of $k(s)$,  a subquotient of the Gau{\ss}-Manin $F$-overconvergent isocrystal 
  $R^if_{s*} (\sO_{Y_s}/ \bar \Q_p)$ for some $i$. 
  \end{corollary}
  \begin{proof}
  Let $(E,\nabla)$ be an irreducible rigid connection. Recall from the introduction that its determinant is finite.  Then 
  by the classical Lefschetz theorem \cite{Lef50}, $(E,\nabla)|_C$  is irreducible, and of course has finite determinant. 
  By Theorem \ref{thm:Fstr},   $M_s:=(E_{K(s)}, \nabla_{K(s)})|_{C_{K(s)}}$ is an irreducible isocrystal with Frobenius structure with  finite determinant. By \cite[Theorem 4.2.2]{Abe18}, $p_1^*M^\vee _s\otimes p_2^*M_s$ on $C_s\times_s C_s$ is a subquotient of $R^ig_{s!} ( \sO_{Cht}/\bar \Q_p)$ where $g_s: Cht_s \to C_s\times_s C_s$ is the Shtuka stack. 
  According to \cite[Corollary 2.3.4]{Abe18}, which states that an admissible stack admits a proper surjective   and generically finite cover by a smooth and projective scheme, we can realize $p_1^*M^\vee _s\otimes p_2^*M_s   $ as a subquotient of 
  some $R^nh_{s*} ( \sO_{Cht}/\bar \Q_p)$ for some projective morphism $h_s: Y_s \to C_s\times_sC_s$, with $Y_s$ smooth projective.
Thus there is an inseparable cover $\sC\xrightarrow{\tau} C_s\times_sC_s$ and a factorization $h_s: Y_s \xrightarrow{\varphi}  \sC \to C_s\times_sC_s$ such that $\varphi$ is generically smooth course projective. 
  There is then a dense open $\sU\subset  \sC$ such that $R^n \varphi_* ( \sO_{\varphi^{-1}(\sU)}/\bar \Q_p)$ satisfies base change.  Fixing a point $(x_s,y_s)\in \sU$, one defines a dense open $C_s^0\subset C_s$ such that $\tau^{-1}(x_s\times C_s^0)\subset \sU$. Set $\psi: V=h_s^{-1}(x_s\times C^0_s) \to C^0_s$. Then base change implies that 
 $ M_s $ is a subquotient of  $R^i \psi_*  (\sO_V/\bar \Q_p)$. 
  \end{proof}

\begin{remark}  \label{rmk:flf}
As we recalled in Theorem \ref{thm:lsz} (see \cite[p. 3, Theorem 3.2, Variant 2]{LSZ13} for the original reference), one can associate to an $f$-periodic flat connection an \'etale $W(\mathbb{F}_{p^f})$-local system on $X_{K(s)}$.  Thus starting with $(E_i,\nabla_i)$ rigid over $X/\C$, for $i=1,\ldots, N$, choosing $s$, one  constructs  a $p$-adic representation
 \ga{padicrepn}{ \rho_{i,s}: \pi_1(X_{K(s)})\to GL(r, W(\F_{p^f})).}
These representations will be studied in the following subsection.
\end{remark}

\section{The $p$-adic representation associated to a rigid connection}
The aim of this section is to prove that the representations $\rho_{i,s}$ defined in \eqref{padicrepn} are rigid as representation of the geometric fundamental group (see Theorem \ref{thm:monodromy}). 

\begin{definition}\label{rep_rigid}
\begin{enumerate}
\item[(a)] If $A$ is a  field,  a representation of an abstract  group $G$ in $GL(r, A)$ is said to be {\it  absolutely irreducible} if 
the representation $G\to GL(r, A)\to GL(r, \Omega)$ is irreducible, where $\Omega$ is any algebraically closed field containing $A$.   
\item[(b)] Let $G$ and $A$ be as in (a). A projective representation $G \to PGL(r,A)$ is said to be \emph{absolutely irreducible} if for every embedding of $A$ into an algebraically closed field $\Omega$ the composition $G \to PGL(r,\Omega)$ is irreducible as a projective representation.
\item[(c)] If $G$ is finitely generated, one defines the moduli scheme  $\MB(G, PGL(r))$  of  $PGL(r)$-representations of $G$, which is also a  coarse moduli  scheme of finite type defined over $\Z$. An isolated point  is called {\it rigid}.   
\end{enumerate}
\end{definition}

The definition above refers to abstract representations. Below we explain how to deal with continuous representations.

\begin{definition}\label{def:rig_profinite}
Let $\Gamma$ be a profinite group, and $\rho\colon \Gamma \to GL(r, F)$ a continuous  and absolutely irreducible  representation,  where $F$ is a topological field. Every finite-dimensional $F$-algebra $A$ inherits a canonical topology from $F$. One denotes by 
$$\mathsf{Def}_{\rho}\colon \mathsf{Art}_{F} \to \mathsf{Sets}$$
the functor
sending a  finite-dimensional  commutative local $F$-algebra $A$ to the set of isomorphism classes of continuous representations $\rho'\colon\Gamma \to GL(r, A)$,   such that there exists a finite field extension $F'/F$, an $F$-morphism $A \to F'$, and an isomorphism $(\rho')_{F'} \simeq \rho_{F'}$.
We say that $\rho$ is rigid, if $\mathsf{Def}_{\rho}$ is corepresented by $B \in \mathsf{Art}_F$.
\end{definition}

 \begin{definition}
For a local field $F$, and its ring of integers $\Oo_F$, we say that a continuous 
  representation $\rho\colon \Gamma \to GL(r,\Oo_F)$ is \emph{rigid and absolutely irreducible}, if the associated  residual representation $\Gamma \to GL(r,k_F)$ is rigid and absolutely irreducible.
\end{definition}
We freely use the notation of the preceding sections. Recall that we choose a model $X_S$ of $X$ over which all rigid connections are defined (see Proposition \ref{prop:nice_models2}).  The field of functions $\Q(S)$ is by definition  embedded in $\C$.   We denote by $\overline{ \Q(S)} $ its algebraic closure in $\C$. 

\medskip 
The following theorem is the main result of this section. We denote by $K(s)$ the local field given by the fraction field of the Witt ring $W(k(s))$ where $s$ is a closed point of $S$.

\begin{theorem}\label{thm:monodromy}
For $s \in S$ a closed point and $\Spec W(k(s)) \to S$ we have that  $\rho_{i,s}|_{\pi_1(X_{\overline{K}_s})}$ is absolutely irreducible and rigid.
\end{theorem}

We start with general facts. We emphasize that the following lemma is based on Definition \ref{def:rigid} of rigidity.
\begin{lemma} \label{lem:glpgl}
Let $\Gamma$ be a finitely generated abstract group, and $K$ be an algebraically closed field of characteristic $0$. We denote by $\rho\colon \Gamma \to GL(r, K)$ an irreducible representation. Then $\rho$ is rigid, if and only if the corresponding projective representation $\rho^{\rm proj}\colon \Gamma \to PGL(r, K)$ is rigid.
\end{lemma}

\begin{proof}
We show first that rigidity of $\rho^{\rm proj}$ implies rigidity of $\rho$. We assume by contradiction that $\rho$ is not rigid. Then there exists a discrete valuation ring $R$ over $K$,  with residue field $K$, thus $K\xrightarrow{\iota} R\xrightarrow{q} K$, and 
a representation $\tilde{\rho}\colon \Gamma \to GL(r, R)$ such that the diagram
\[
\xymatrix{
\Gamma \ar[r] \ar[rd] & GL(r, R) \ar[d]^q \\
& GL(r, K)
}
\]
commutes. This amounts to a non-trivial deformation. Since $\rho^{\rm proj}$ is rigid, the associated projective representation $\tilde{\rho}^{\rm proj}$ has to be constant. In particular we conclude that $\tilde{\rho}^{\rm proj} \otimes \overline{ {\rm Frac}( R)}$ is equivalent to $\rho^{\rm proj}$ after base change. This implies that there exists a character $$\chi\colon \Gamma \to  (\overline{{\rm Frac}( R)})^\times,$$ such that 
$$ \iota\circ \rho \otimes  \overline{{\rm Frac}( R)} \simeq (\tilde{\rho} \otimes \overline{{\rm Frac} ( R)}) \otimes \chi.$$
Taking determinants, we see that $\chi^{r}$ is trivial, which implies that $\chi$ is already defined over $K$. By irreducibility of $\rho$ over $R$, there is an isomorphism 
$$ \iota\circ \rho \simeq (\tilde{\rho} \otimes {\rm Frac} (R)) \otimes \chi$$
defined over ${\rm Frac} (R)$.
We conclude the proof by observing that $\iota \circ \rho \simeq \tilde{\rho} \otimes \chi$ implies 
${\rm Tr}(\rho(g)) =\chi(g) {\rm Tr} (\tilde \rho (g) )$ for all $g\in \Gamma$.
Thus we have $\chi(g)=1$. 

\medskip

Vice versa, let us assume that $\rho$ is rigid. Let $\chi = \det \rho\colon \Gamma \to K^{\times}$ be the determinant of $\rho$. As above we consider a non-trivial deformation $\widetilde{\rho}^{\rm proj}\colon \Gamma \to PGL_n(R)$ of $\rho^{\rm proj}$. The obstruction of lifting $\widetilde{\rho}^{\rm proj}$ to a homomorphism $\widetilde{\rho}\colon \Gamma \to GL_n(R)$ with $\det \widetilde{\rho} \simeq \chi$ lies in $H^2(\Gamma,\mu_n(R))$. Since $\mu_n(R) = \mu_n(K)$, and the obstruction vanishes over the residue field (indeed, $\rho^{\rm proj}$ is the projectivization of $\rho$), we see that the obstruction vanishes also over $R$.
This shows the existence of an $R$-deformation of $\rho$, which is non-trivial, since the associated projective representation is non-trivial.
\end{proof}

Recall that for a profinite ring $A$, an abstract representation $\rho:  \pi_1^{\rm top}(X)\to GL(r, A)$ factors through a continuous profinite representation $ \hat \rho\colon \pi_1(X)\to GL(r, A)$, similarly for a $PGL(r, A)$ representation. The next lemma gives a criterion for rigidity.

\begin{lemma}\label{lemma:cont_rigid} 
Let $\rho\colon  \pi_1^{\rm top}(X)\to GL(r, \mathbb{F}_q)$ be an absolutely irreducible representation. Let   $[\rho]\in \MB(X/\C, {\rm det}(\rho),r)(\mathbb{F}_q)$ be its moduli point. Then   $[\rho]$ lies in $\MBrig(X/\C, {\rm det}(\rho),r)(\mathbb{F}_q)$ if and only if the continuous representation
$\hat \rho$ is rigid.  The analogous assertion holds for projective representations.
\end{lemma}

\begin{proof}
We deduce this from the fact that for a finite-dimensional $\mathbb{F}_q$-algebra $A$, one has a canonical bijection between continuous morphisms 
$\pi_1(X) \to GL(r,A)$
and morphisms $\pi^{\rm top}_1(X) \to GL(r,A)$. This shows that $\mathsf{Def}_{\hat{\rho}}(\mathbb{F}_q)$ is represented by the formal scheme  which is  the formal completion of  $\MB(X/\C, {\rm det}(\rho),r)  \otimes_{\Z} \F_p $ at the point $[\rho]$. The latter is equivalent to the spectrum of an artinian $\mathbb{F}_q$-algebra if and only if $[\rho]$ is an isolated point. That is, if and only if $\rho$ is rigid. Definition \ref{def:rig_profinite} of rigidity for continuous representations of profinite groups allows us to conclude the proof.
\end{proof}

Let $K_{\rm p-mon}$ be a number field such that the topological monodromy of every rank $r$ irreducible rigid projective representation has values in $PGL(r, K_{\rm p-mon})$, and  $K_{\rm mon}$ be  a number field such that the topological monodromy of every rank $r$ irreducible rigid representation has values in $GL(r, K_{\rm mon})$. Since there are only finitely many irreducible rigid representations, and $\pi_1^{\rm top}(X)$ is finitely presented, there exists $M \in \Oo_{K_{\rm p-mon}}$ such that every such representation is defined over $PGL(r, \Oo_{K_{\rm p-mon}}[M^{-1}])$. We write $\mathbb{O}_{\rm p-mon,M} 
 = \prod_{\nu} \Oo_{K_{\rm p-mon},\nu}$, where $\nu$ ranges over the places of $K_{\rm p-mon}$ such that $\nu(M) = 1$. As a topological group, $\mathbb{O}_{\rm p-mon,M}
$ is profinite.

\begin{proposition}[Simpson, \cite{Sim92}, Theorem 4]\label{prop:simpson}
Let ${\rho}\colon \pi_1^{\rm top}(X) \to PGL(r, \Oo_{K_{\rm p-mon}})$ be an absolutely irreducible rigid $PGL(r)$-representation. 
Then there is a finite Galois extension $L/ \Q(S)$ such that $\rho \otimes \mathbb{O}_{\rm p-mon,M}
$ extends to a projective representation
\[
\xymatrix{
\pi_1^{\rm top}(X) \ar[r] \ar[rd]_\rho  & \pi_1(X_L) \ar@{..>}[d] \\
& PGL(r, \mathbb{O}_{\rm p-mon,M}).
}
\]
\end{proposition}

\begin{proof}
Let  $\hat \rho\colon \pi_1(X)
\to PGL(r, \mathbb{O}_{\rm p-mon,M} )$ be the profinite representation associated to $\rho$. We apply Simpson's theorem {\it loc. cit.}, with the slight difference that we use here directly a projective representation in the assumption. 
For the reader's convenience, we sketch Simpson's argument in this context.  We choose finitely many generators $A_1,\dots,A_N$ of $\pi_1^{\rm top}(X)$ and use them to embed 
$$\mathcal{R}(\pi_1^{ \rm top}(X))(  \mathbb{O}_{\rm p-mon,M}) \hookrightarrow PGL(r, \mathbb{O}_{\rm p-mon,M} ))^{N},$$ 
where $\mathcal{R}(\pi_1^{\rm top}(X))(   \mathbb{O}_{\rm p-mon,M}
) $ is the set of $\mathbb{O}_{\rm p-mon,M} $-points of the scheme of representations defined by the image of the $A_i$ satisfying the relations of the topological fundamental group.
Thus $\mathcal{R}(\pi_1^{ \rm top}(X))$ is endowed with the profinite topology. 
To $\gamma \in \Gamma = {\rm Gal}(\overline{\Q(S)}/ \Q(S))$ one assigns the representation
$\rho^\gamma\colon  \pi_1(X) \to PGL(r,  \mathbb{O}_{\rm p-mon,M}), \ c\mapsto \rho(\gamma c \gamma^{-1})$. 
Continuity is checked as in   {\it loc. cit.} As $\rho$ is rigid, there is an open subgroup $U\subset \Gamma$ such that for $\gamma \in U$, the representation $\rho^\gamma$ is isomorphic to 
$\hat \rho$. 
We set $L= \overline{\Q(S)}^U$.  This yields the factorization\[
\xymatrix{
\pi_1(X) \ar[r] \ar[rd]_{\hat \rho} & \pi_1(X_L) \ar@{..>}[d] \\
& PGL(r,  \mathbb{O}_{\rm p-mon,M} ). 
}
\]
\end{proof}

The Lan--Sheng--Zuo correspondence only relates periodic Higgs bundles with crystalline representations. In particular we do not know that the representation thereby assigned to a rigid Higgs bundle is again rigid. This problem is solved in the sequel by relating the Lan--Sheng--Zuo correspondence to Faltings's Simpson correspondence established in \cite[Theorem 5]{Fal05}. Recall that Faltings defines a category of generalized representations (of the geometric \'etale fundamental group) of a $p$-adic scheme, see \cite[Section 2]{Fal05}, {and defines the notions of \emph{small} generalized representations and $\emph{small}$ Higgs bundles (see \emph{loc. cit.}).}

\begin{theorem}[Faltings]\label{thm:Faltings}
Let $K$ be a local field with ring of integers $V$, let $X$ be a proper $V$-scheme with toroidal singularities. There exists an equivalence of categories between small Higgs bundles on $X_{\bar{K}}$ and small generalized $K$-representations of $\pi_1^{\rm{\acute{e}t}}(X_{\bar{K}})$.
\end{theorem}

\medskip

We fix a closed point $s \in S$ and consider a morphism $\Spec W(k(s)) \to S$. The fraction field of $W(k(s))$ will be denoted by $K(s)$.
As in Remark \ref{rmk:flf} we list the representations $\rho_{i,s}\colon \pi_1(X_{\mathbb{Q}_p}) \to GL(r, W(\mathbb{F}_{p^{f_{i,s}}}))$ corresponding to rank $r$ irreducible rigid Higgs bundles $(E_i,\theta_i)$ on $X_{W(k(s))}$.   We let $\sigma^{m}(E_i,\theta_i)$  denote the periodic Higgs bundle corresponding to $\sigma^m(\rho_{i,s})$ by the Lan--Sheng--Zuo correspondence.
That is, $\sigma$ is the shift operator defined in \eqref{sigma}  on Higgs bundles, which corresponds to the Frobenius action on the coefficients $W(\F_{p^{f_{i,s}}})$ for $\rho_{i,s}$. 

We recall that $\sigma$ denotes Lan--Sheng--Zuo's shift functor for periodic Higgs bundles. On the level of crystalline representations it corresponds to the Frobenius-twist of a representation. Recall that we denote by $f_{i,s}$ the smallest positive integer such that $\sigma^{f_{i,s}}(E_i,\nabla_i) \simeq (E_i,\nabla_i)$.

\begin{lemma}\label{lem:F_p-adic}
There exists an $m_{i,s}$ such that under Faltings's $p$-adic Simpson correspondence the representation of the geometric fundamental group $\rho_{i,s}|_{\pi_1(X_{\overline{K}_s})}$ is isomorphic to $\sigma^{m_{i,s}}(E_i,\theta_i)$ as representations defined over 
${\rm Frac}(W(\mathbb{F}_{p^{f_{i,s}}}))$.
\end{lemma}

\begin{proof} We fix a $\rho_{i,s}$ which we for short denote by $\rho$. 
We define the $\mathbb{Z}_p$-representation
$$\rho_{\rm big}=\bigoplus_{i=0}^{f_{i,s}-1}\sigma^i\rho $$
of $\pi_1(C_{\Q_p})$. 
 It corresponds to the Higgs bundle $(E_{\rm big},\theta_{\rm big})$ of period $1$ by means of Lan--Sheng--Zuo's correspondence.

We compare this correspondence with Faltings's one. First we remark that Faltings's correspondence relates small generalized representations of $\pi_1(X_{\overline{K}_s})$ and small Higgs bundles on $X_{\overline{K}_s}$ (see Theorem \ref{thm:Faltings}). The class of representations is preserved by deformations inside generalized representations, which allows us to test rigidity. Furthermore, every Higgs bundle with nilpotent Higgs field is small (smallness is defined via the characteristic polynomial in the $\mathbb{Q}_p$-theory), and every $\mathbb{Z}_p$-representation induces a small $\mathbb{Q}_p$-representation.

By definition, $\rho_{\rm big}$ corresponds to a Frobenius crystal, and therefore, by \cite[Section 5, Ex.]{Fal05},
Faltings's correspondence associates to $\rho_{\rm big} \otimes \mathbb{Q}$ the Higgs bundle $(E_{\rm big},\theta_{\rm big})$. This example  is thus compatible with the Lan--Sheng--Zuo correspondence.

Furthermore, we have an isomorphism 
$$\rho_{\rm big}|_{\overline{\mathbb{Q}}_p} \otimes_{\Z} \mathbb{Q} = \left(\bigoplus_{i=0}^{f_{i,s}-1}\sigma^i\rho|_{\overline{\mathbb{Q}}_p} \right) \otimes_{\Z} \mathbb{Q}   \simeq \bigoplus_{m=0}^{f_{i,s}-1} \rho^F_m.$$
This implies that each factor on the left hand side is isomorphic to a unique factor on the right hand side. In particular, we obtain the requested isomorphism
$\rho_{i,s} \otimes \mathbb{Q}_q \simeq \rho^F_{m_{i,s}}.$  
\end{proof}
{
\begin{corollary} \label{cor:f1}
The representation $\rho_{i,s}^{\rm geom} = \rho_{i,s}|_{\pi_1(X_{\bar{\mathbb{Q}}_p})}$ is absolutely irreducible and rigid.
\end{corollary}

\begin{proof} 
It follows from Proposition \ref{Crigid} that $\sigma^m(E_i,\theta_i)$ is again rigid. By virtue of Lemma \ref{lem:F_p-adic} we see that the representation $\rho_{i,s}^{\rm geom} = \rho_{i,s}|_{\pi_1(X_{\bar{\mathbb{Q}}_p})}$ is associated to a stable rigid Higgs bundle under Faltings's correspondence, and hence is rigid and absolutely irreducible.
\end{proof}
}
{
\begin{proof}[Proof of Theorem \ref{thm:monodromy}]
In Corollary \ref{cor:f1} we have shown that the representation $\rho_{i,s}^{\rm geom} = \rho_{i,s}|_{\pi_1(X_{\bar{\mathbb{Q}}_p})}$ is absolutely irreducible and rigid. This concludes the proof of Theorem \ref{thm:monodromy}.
\end{proof}
}
 
In the following we denote by $q$ a power of a prime $p$. We use the shorthand $\mathbb{Z}_q$ for $W(\mathbb{F}_q)$.

\begin{lemma} \label{lem:good}
There exist infinitely many prime numbers $p$, such that 
\begin{itemize}
\item[(a)] every rigid and absolutely irreducible $\mathbb{Z}_q$-representation of $\pi_1^{\rm top}(X)$ of rank $r$ and determinant $L$ is defined over $\mathbb{Z}_p$, and similarly for rigid absolutely irreducible projective representations (in particular every place $\nu$ over $p$ in $K_{\rm mon}$ and $K_{\rm p-mon}$ splits completely);
\item[(b)] every rigid and absolutely irreducible representation $\rho\colon \pi_1^{\rm top}(X) \to GL(r,\mathbb{F}_p)$ is obtained as reduction modulo $p$ of a representation $\pi_1^{\rm top}(X) \to GL(r,\Oo_{K_{\rm mon}}[M^{-1}])$, and similarly for projective representations;
\item[(c)] there exists a closed point $s \in S$ with $k(s) = \mathbb{F}_p$, which is the specialization of a morphism $\Q(S) \to \mathbb{Q}_p$;
\item[(d)] for every $s$ as in (b), a rigid and absolutely irreducible projective representation of $\pi_1(X_{\overline{\mathbb{Q}}_p}) \to PGL_n(\mathbb{Z}_p)$ descends to $\pi_1(X_{\mathbb{Q}_p})$.
\end{itemize}
We call such a closed point $s$ \emph{good}, if in addition the order of $L$ is prime to $p$.
\end{lemma}

\begin{proof}
We write $S = \Spec R$.  We denote by $R_1 = \mathbb{Z}[\alpha_1,\dots,\alpha_m]$ a finitely generated subalgebra of $\mathbb{C}$ containing $R$, $\Oo_{K_{\rm mon}}$ and $\Oo_{K_{\rm p-mon}}[M^{-1}]$, as well as the normalization of $S$ inside all of the (finitely many) field extensions $L/\mathbb{Q}(S)$ constructed in 
Proposition \ref{prop:simpson}. By Cassels's embedding theorem \cite[Theorem I]{C76}, there are infinitely many prime numbers $p$ such that the fraction field $\mathbb{Q}(\alpha_1,\dots,\alpha_m)$ can be embedded in $\mathbb{Q}_p$, and the generators $\alpha_i$ are sent to $p$-adic units. For such a $p$, 
the induced morphism  $\Oo_{K_{\rm mon}} \to \mathbb{Z}_p$ is well-defined and injective.  This shows (a). 

Claim (b) is automatic by choosing very large prime numbers: since the variety of $\pi_1^{\rm top}(X)$-representations is of finite type over $\mathbb{Z}$, the subscheme of rigid representations is finite over a dense open of ${\Spec}\Z$, thus there can only be finitely many primes where isolated points exist that do not have a $K_{\rm mon}$-model. 

Moreover we have a non-trivial morphism $R \to \mathbb{Z}_p$, hence the composition $R \to \mathbb{F}_p$ defines the required $\mathbb{F}_p$-rational point $s$ in (b). 

Claim (c) follows from (b) and Proposition \ref{prop:simpson}. At first we choose an abstract isomorphism of fields $\mathbb{C} \simeq \overline{\mathbb{Q}}_p$, and view a rigid and absolutely irreducible projective representation $\rho$ of $\pi_1(X_{\overline{\mathbb{Q}}_p})$ as one of $\pi_1(X)$.

We know that $\rho$ is obtained from a rigid absolutely irreducible representation defined over $\Oo_{\rm p-mon}[M^{-1}]$. By Proposition \ref{prop:simpson} the associated projective $
\mathbb{O}_{\rm p-mon,M} $-representation descends to $X_{L}$. By tensoring along $
\mathbb{O}_{\rm p-mon,M}   \to \mathbb{Z}_p$ (using that $p$ splits completely in $K_{\rm p-mon}$ by (a)), we see that the projective representation $\rho$ itself descends to $X_L$. However, by construction our $p$ also splits in $L$ and therefore we see that $\rho$ descends to $X_{\mathbb{Q}_p}$.
\end{proof}

\begin{corollary}\label{cor:f_p_one}
{For every good closed point $s \in S$ with ${\rm char}(k(s)) = p$ we have for every rigid stable Higgs bundle $(V,\theta)$ defined over $X_s$ that $\sigma(V,\theta)$ and $(V,\theta)$ are isomorphic as $PGL(r)$-Higgs bundles.} 
\end{corollary}

\begin{proof}
Recall that $\mathbb{Z}_{p^{f_{i,s}}}$ is the ring of definition of $\rho_{i,s}$. We have seen in the proof of Theorem \ref{thm:monodromy} that $\rho_{i,s}^{\rm proj}$ is isomorphic to a projective representation defined over $\mathbb{Z}_p$. This implies the assertion.
\end{proof}

\section{Rigid connections with vanishing $p$-curvature have unitary monodromy}\label{sec:unitary}
The aim of this section is to prove Theorem \ref{thm:unitary}, which asserts that a rigid flat connection $(E,\nabla)$ satisfying the assumptions of the $p$-curvature conjecture has unitary monodromy.  
\medskip

We use that a rigid flat connection has the structure of a complex variation of Hodge structure $(E,F^m,\nabla)$ (by \cite[Lemma 4.5]{Sim92}). A complex variation of Hodge structure has unitary monodromy if and only if its Kodaira--Spencer class $\mathrm{gr}(\nabla) \colon \mathrm{gr}(E) \to \mathrm{gr}(E) \otimes \Omega_X^1$ vanishes.

\begin{theorem}
Let $(E_S,\nabla_S)$ be an $S$-model of a rigid connection such that for all closed points $s \in S$,  the connections $(E_s,\nabla_s)$ have vanishing $p$-curvatures. Then $(E_{\mathbb{C}},\nabla_{\mathbb{C}})$ is a unitary connection.
\end{theorem}

\begin{proof}
We choose a good closed point $s \in S$ (see Lemma \ref{lem:good}). By virtue of Corollary \ref{cor:f_p_one} we have $f_{i,s}^{proj} = 1$ for all $i$.
By Corollary \ref{cor:f_p_one}, we know that  $\sigma(V_s,0)$ and $(\mathrm{gr}(E_s),\mathrm{gr}(\nabla_s))$ are equivalent as projective Higgs bundles (this follows from the definition of the Higgs-de Rham flow). This shows that $\mathrm{gr}(\nabla_s) = \omega \cdot{} \id_{\mathrm{gr}(E_s)}$, where $\omega$ is a $1$-form on $X_s$. However, we also know that $\mathrm{gr}(\nabla_s)$ is nilpotent, hence we must have $\mathrm{gr}(\nabla_s) = 0$ for every good $s \in S$.

Since there are infinitely many prime numbers $p$ such that there exists a good closed point $s \in S$ (Lemma \ref{lem:good}), we conclude that the Kodaira--Spencer class vanishes everywhere on $X_S$. Vanishing of the Kodaira--Spencer class implies that $\nabla$ is a unitary connection.
\end{proof}

Under certain circumstances the result above can be used to deduce finiteness of the monodromy. This is the case if the monodromy of the flat connection is known to be \emph{strongly integral} (defined below).

\begin{remark}
According to Simpson's integrality conjecture, a rigid connection is expected to have integral monodromy, that is, the monodromy representation is isomorphic to a representation $\rho\colon \pi^{\rm top}_1(X,x) \to GL_n(\bar{\mathbb{Z}})$. Here, $\bar{\mathbb{Z}}$ denotes the ring of algebraic integers. We emphasize that this is not the same as \emph{strong integrality}, which amounts to the existence of an isomorphism with a representation $\pi^{\rm top}_1(X,x) \to GL_n(\mathbb{Z})$. While it is true that 
$$\text{strong integrality \& unitary} \Rightarrow \text{finite monodromy},$$
it does not hold that ``integrality and unitary" implies finite monodromy. We give a counterexample below.
\end{remark}

\begin{example}
Let $\alpha \in \bar{\mathbb{Z}} \setminus \mu_{\infty}$ be an algebraic integer which is not a root of unity such that $|\alpha| = 1$ (see for instance \cite[Theorem 2]{daileda} for a proof of existence). 

Let $\Sigma$ be an orientable Riemann surface of genus $g \geq 1$, and $x \in \Sigma$. We define a representation $\rho\colon \pi_1^{\rm top}(\Sigma,x) = \langle a_1,\dots,a_{2g} | [a_1,a_2]\cdots{}[a_{2g-1},a_{2g}] \rangle \to GL_1(\bar{\mathbb{Z}})$ as follows:
$$\rho\colon a_1 \mapsto \alpha,\text{  } a_i \mapsto 1 \text{  } \forall i > 1.$$
This representation is integral and unitary by construction. However, it cannot be of finite monodromy, since otherwise $\alpha$ would be a root of unity.
\end{example}

\section{A remark on cohomologically rigid connections and companions} \label{sec:comp}

Recent years saw various breakthroughs on Deligne's companion conjecture (\cite[Conjecture 1.2.10]{Del80}), see \cite{AE16, Dri16} and the brief overview given below. While the $\ell$-adic theory can now be considered to be complete, the existence of \emph{les petits camarades cristallins} is open, except for dimension $1$ (see Abe's \cite{Abe18}). This section serves as an extended remark establishing the existence of these $p$-adic companions for $F$-isocrystals stemming from cohomologically rigid flat connections. This result provides further evidence for Simpson's Conjecture \ref{conj:mot}. In line with the main narrative of this article, this observation follows from a counting argument.

\medskip

We use the standard notations. We choose a closed point $s \in S$.
  Given a rigid connection $(E,\nabla)$ with finite order determinant $L$, we constructed the isocrystal with Frobenius structure on $X_{K_v}$  in Theorem \ref{thm:Fstr}.  For notational convenience, in this section, we denote it by $\sF$.

  \medskip

     For the reader's convenience, we summarise the defining properties of $\ell$-adic companions.
   Let $\mathcal{E}$ be an irreducible isocrystal with Frobenius structure  on a smooth variety $Y$ of finite type over $\mathbb{F}_q$ of characteristic $p>0$. To every closed point $y \in Y$,  one attaches the characteristic polynomial $P_{\sE,y}(t)={\rm det}(1-tFr_y|\sE_y)  \in \bar{\mathbb{Q}}_p[t]$, where $Fr_y$  is the  absolute Frobenius  at $y$ acting on $\mathcal{E}_y:=i_y^*\sE$, and where $i_y\colon y\to Y$ is the closed embedding. 
Similarly for a lisse $\bar \Q_{\ell}$-sheaf $\sV$ on $Y$,  for every closed point $y \in Y$, one attaches  the characteristic polynomial $P_{\sV,y}[t] ={\rm det}(1-tF_y|\sV_{\bar y})   \in \bar{\mathbb{Q}}_\ell[t] ,$  where $F_y$ is the geometric Frobenius  $F_y$ acting on $\sV_{\bar y}$, where $\bar y\to y$  an $\bar \F_p$-point above $y$. 

\begin{definition}[See \cite{AE16}, Definition 1.5, \cite{Dri16}, Section 7.4.]
\begin{enumerate}
   \item Let $\tau\colon \bar \Q_p \xrightarrow{\simeq }\bar\Q_{\ell}$ be an  abstract isomorphism of fields. We say that an irreducible lisse $\bar \Q_\ell$-sheaf $\sV$ is a $\tau$-companion of an irreducible isocrystal $\sE$ with Frobenius structure, or equivalently that an irreducible isocrystal $\sE$ with Frobenius structure is a $\tau^{-1}$-companion of an irreducible lisse $\bar \Q_\ell$-sheaf $\sV$ 
  if for every closed  point $y \in Y$,  one has an equality of characteristic polynomials $$\tau (P_{\sE,y}(t)) = P_{\sV,y}[t] \in \bar \Q_\ell[t].$$ 
  \item Let $\tau \colon \bar \Q_p \xrightarrow{\simeq }\bar\Q_{p}$ be an abstract field isomorphism. We say that an irreducible isocrystal $\sE'$ with Frobenius structure is a $\tau$-companion of    an irreducible isocrystal $\sE$ with Frobenius structure  if   for every closed  point $y \in Y$,  one has an equality of characteristic polynomials $$\tau (P_{\sE,y}(t)) = P_{\sE',y}[t] \in \bar \Q_\ell[t].$$   
  \item Let $\tau \colon \bar \Q_\ell \xrightarrow{\simeq }\bar\Q_{\ell'}$ be an abstract field isomorphism.  We say that an irreducible lisse $\bar \Q_{\ell '}$-sheaf $\sV'$ is a $\tau$-companion of an irreducible  lisse $\bar \Q_{\ell }$-sheaf $\sV$ if for every closed  point $y \in Y$,  one has an equality of characteristic polynomials $$\tau (P_{\sV,y}(t)) = P_{\sV',y}[t] \in \bar \Q_\ell[t].$$ 
   
   \end{enumerate}
  \end{definition}
   
   We shall use in the sequel that  $\tau$-companions exist by 
     \cite[Theorem 4.2]{AE16}  (see  \cite[Theorem 7.4.1]{Dri16} for a summary, and  \cite{Ked17} for later work in progress) for $\tau 
   \colon \bar \Q_p \xrightarrow{\simeq }\bar\Q_{\ell'}$, for irreducible objects with finite determinant.  
   They also exist 
   by Drinfeld's theorem  \cite[Theorem 1.1]{Dri12} for $\tau 
   \colon \bar \Q_\ell \xrightarrow{\simeq }\bar\Q_{\ell'}$ 
   for irreducible objects with finite determinant.  We shall not use Drinfeld's $\ell$-to-$\ell '$ existence theorem. 
     By    \v{C}ebotarev's density theorem and Abe's \v{C}ebotarev's theorem  \cite[Proposition A.3.1]{Abe18},
{companions are unique up to isomorphism},  and companions of two non-isomorphic objects are non-isomorphic, and the companion of an order $d$ rank $1$ object is an order $d$ rank $1$ object.
 The general conjecture is that companions exist for all $\tau$. The remaining cases are for $\tau\colon \bar \Q_p\to \bar \Q_p$  and   $\tau\colon \bar \Q_\ell\to \bar \Q_p$. We will see that for the cohomologically rigid case, existence of $p$-to-$p$  and $\ell$-to-$p$ companions can be shown.

 \medskip 
 
 Recall   that an irreducible  connection $(E,\nabla)$ with determinant $L$ on $X$ over $\C$  is called {\it cohomologically rigid}  if $H^1(X, \Eend^0(E,\nabla))=0$, that is $(E,\nabla)$   is rigid and in addition its moduli point  $[(E,\nabla)]\in \MdR(X/\C, L,r)$ is smooth.  Here $L$ is torsion of order $d$. 
 See \cite[Section 2]{EG17} for a general discussion of the notion even in the non-proper case. In our situation, where $X$ is proper, it is straightforward to see that $H^1(X, \Eend^0(E,\nabla))=0$ is the Zariski tangent space  of $ \MdR(X/\C, L,r)$ at the moduli point.  By base change for de Rham cohomology one has 
 $H^1(X_{ \bar \Q_p}, \Eend^0((E,\nabla)))=0$, and this last  group is equal to crystalline cohomology over $\bar \Q_p$ of the isocrystal $\sF$ (see Corollary \ref{cor:isoc}), thus 
  $$H^1_{\rm crys} (X_{\bar s}, \Eend^0(\sF)) \otimes_{W(\bar \F_p)} \bar \Q_p=0. $$ 

\begin{definition}For $s$ a closed point of $S$ in Theorem \ref{thm:Fstr}  we denote by $\tilde{k}(s) \supset k(s)$ a finite extension such that every rigid stable rank $r$ flat connection $(E,\nabla)$ gives rise to an $F$-isocrystal (see Theorem \ref{thm:Fstr}).
\end{definition}

Let $\sS(s, p, r, d)$ be the finite set of  isomorphism classes of isocrystals $\sF$ of rank $r$ and determinant of order $d$ obtained this way, with $d$ prime to the characteristic $p$ of $s$,  such $H^1_{\rm crys} (X_{\bar s}, \Eend^0(\sF)) \otimes_{W(\bar \F_p)} \bar \Q_p=0. $
For a prime $\ell \neq p$, we denote by  $\sS(s, \ell, r, d)$ the finite set of isomorphism classes of  irreducible $\bar \Q_\ell$-adic sheaves $\sV$ of rank $r$ and order $d$ determinant such that 
$H^1 (X_{\bar s}, \Eend^0(\sV)) =0. $ We denote by $S(r,d)$ the finite set of isomorphism classes of irreducible rank $r$ cohomologically rigid connections with order $d$ determinant on $X$.
 \begin{thm} \label{thm:comp}
\begin{itemize}
\item[0)] Theorem \ref{thm:Fstr} defines a bijection between $\sS(r,d)$ and $\sS(s,p,r.d).$ 
\item[1)]  Let $ \tau\colon \bar \Q_p \to \bar \Q_\ell$ be an isomorphism. 
The companion correspondence  for $\tau$ establishes a bijection between $\sS(s, p, r, d)$
and $\sS(s, \ell, r, d)$, defining for $\tau^{-1}$  the companions of the elements in $\sS(s, \ell, r, d)$.
\item[2)]  Let  $ \sigma\colon \bar \Q_p \to \bar \Q_p$ be an isomorphism. Then $\sigma$-companions of elements in 
$\sS(s, p, r, d)$ exist and the companion correspondence  for $\sigma$ establishes a permutation of $\sS(s, p, r, d)$.
\end{itemize}
\end{thm}
The proof occupies the rest of this section. We fix a prime $\ell \neq p$    and write 
$$\sigma\colon  \bar \Q_p\xrightarrow{\tau} \bar \Q_\ell \xrightarrow{\tau^{'-1}} \bar \Q_p$$ 
for a choice of $\tau$. 
Thus
 2) follows directly from 1) by applying 1) to $\tau$ and $\tau'$.

We denote by $\sV$ the $\tau$-companion  of $\sF \in \sS(s,p,r,d)$.   It corresponds to an irreducible continuous representation  $ \sV_s\colon \pi_1(X_s)\to GL(r, \bar \Q_\ell)$ with finite determinant, thus precomposing by the (surjective) specialization homomorphism $sp\colon \pi_1(X_{K(s)})\to \pi_1(X_{s})$, it defines an irreducible $\ell$-adic lisse sheaf    $\sV_{K(s)} \colon \pi_1(X_{K(s)}) \to GL(r, \bar \Q_\ell)$  on $X_{K(s)}$, and the underlying geometric representation $\sV_{\bar \Q_p}\colon \pi_1(X_{\bar \Q_p}) \to GL(r, \bar \Q_\ell)$ on $X_{K(s)}$.

\begin{proposition} \label{prop:cohrig} 
\begin{itemize}
\item[1)] 
The representation $\sV_{\bar \Q_p}$ is irreducible. 
\item[2)] We have the vanishing result
$$H^1(X_{\bar \Q_p}, \Eend^0(\sV_{\bar \Q_p}))=0.$$
\item[3)] The order of ${\rm det}(\sV_{\bar \Q_p})$ is $d$. 
\end{itemize}

\end{proposition}
\begin{proof}
We prove 1). 
Since $sp$ restricts to the specialization $\pi_1(X_{\bar \Q_p})\to \pi_1(X_{\bar s})$, where $\bar s\to s$ is a geometric point with residue field $\bar \F_p$, and is still surjective, we just have to show that $\sV_s$ is geometrically irreducible. Theorem \ref{thm:monodromy} together with the construction of 
Theorem \ref{thm:Fstr} shows that the isocrystal is compatible with  finite base change $s'\to s$.  Since an irreducible $\ell$-adic sheaf which is not geometrically irreducible splits over a finite base change $s'\to s$ (see e.g. \cite[(1.3.1)]{Del12}), this shows that $\sV_s$ is geometrically irreducible. 

\medskip

We turn to the proof of 2).
Let us consider the $L$- functions  $L(X_s,  \Eend^0(\sV_{\bar s}))$  and  $L(X_s, \Eend^0(\sF))$
for the lisse $\bar{\mathbb{Q}}_{\ell}$-sheaf $\Eend^0(\sV_s)$ and for the isocrystal with Frobenius structure $\Eend^0(\sF)$ (\cite[5.2.3]{Del73} and 
\cite[4.3.2]{Abe18}).  The product formula for these $L$-functions implies that they are equal.  On the other hand, 
as $\sV_{\bar s}$  and $\sF$, thus a fortiori $\Eend^0(\sV_{\bar s})$  and $\Eend^0(\sF)$ have weight $0$ (see  \cite[Proposition VII.7 (i)]{Laf02}, corrected in \cite[Cor. 4.5]{EK12}), the dimension of  
$H^1(X_{\bar s}, \Eend^0(\sV_{\bar s}))$ over $\bar \Q_\ell$ and of $H^1_{\rm crys}(X_{\bar s}, \Eend^0(\sF)) \otimes_{\Q_p} \bar \Q_p$ over $\bar \Q_p$  are computed as the number of weight $1$ eigenvalues counted with multiplicities in the $L$-function (see \cite[Lemma 3.4]{EG17} for a more general purity argument). This shows that both are the same.  On the other hand,  one has 
 $H^1_{\rm crys}(X_{\bar s}, \Eend^0(\sF)) = 0$.
 By \cite[Theorem 8.0]{Kat70}, there is a dense open  $S'\hookrightarrow S$ on which one has  base change for de Rham cohomology. Thus we conclude that 
$$ 0={\rm dim}_{\Q_p}H^1_{dR}(X_{\Q_p},  \Eend^0(E,\nabla))= {\rm dim}_{\C}H^1_{dR}(X, \Eend^0(E,\nabla)).  $$
Therefore, we have
${\rm dim}_{\bar \Q_\ell} H^1(X_{\bar s}, \Eend^0(\sV_{\bar s}))=0$.  

\medskip
It remains to  see that the specialization homomorphism
 $$H^1(X_{\bar s}, \Eend^0(\sV_{\bar s})) \to H^1(X_{\bar K(s)}, \Eend^0(\sV_{\bar K(s)}))$$
  is an isomorphism, which is true by local acyclicity and proper base change \cite[Cor. 1.2]{Art73}.
  
  \medskip
  
We now turn to the proof of 3). By definition, the order of ${\rm det}(\sV_{\bar \Q_p})$  is the same as the order of 
   ${\rm det}(\sV_{\bar s})$. Since the order  $d$ of ${\rm det}(\sF)$ does not change after replacing $s$ by a finite extension,  its companion ${\rm det}(\sV_{s})$ is of order $d$. Thus ${\rm det}(\sV_{\bar s})$ is of order $d$.
\end{proof}

\begin{corollary} \label{cor:cohrig} The composite representation $$\rho\colon \pi_1^{\rm top}(X)\to \pi_1(X)=\pi_1(X_{\bar \Q_p})\xrightarrow{\sV_{\bar \Q_p}}  GL(r, \bar \Q_\ell)$$ defines a cohomologically rigid $\bar \Q_\ell$-point of $\MBrig(X/\C, L,r)$, for some $L$ of order $d$, and any cohomologically rigid $\bar \Q_\ell$-point of $\MBrig(X/\C,L,r)$ with determinant of order $d$, arises in this way. 
\end{corollary}
\begin{proof}
Let $A\subset \Z_\ell$ be a subring of finite type such that  $\rho$ factors though $\rho_A\colon \pi_1^{\rm top}(X) \to GL(r, A)$, and $\iota\colon A\hookrightarrow \C$ be a complex embedding. Set $\rho_{\C}=\iota \circ \rho_A$. Then 
$$H^1_{\rm an}(X, \Eend^0(\rho_{\C}))= H^1_{\rm an}( X, \Eend^0(\rho_A))\otimes_A \C. $$  Here the subscript $_{\rm an}$ stands for the analytic topology. On the other hand, 
by comparison between analytic and \'etale cohomology, one has $$H^1_{\rm an}(X, \Eend^0(\rho_A)) \otimes_A \Q_\ell=  H^1_{\rm \acute{e}t}(X, \Eend^0(\rho))=H^1(X_{\bar \Q_p},  \Eend^0(\sV_{\bar \Q_p}))=0.$$
This proves the first part. It also shows that for every closed $s \in S$, the number of the $\sV_{\bar s}$ is at most the number of cohomologically rigid connections over the complex variety $X$. 

\medskip

Vice versa, given all cohomologically rigid  points of $\MdR(X/\C,L,r)$, their restrictions to $X_{K(s)}$ for a closed point $s \in S$ define isocrystals with Frobenius structure (Theorem \ref{thm:Fstr}), which are cohomologically rigid  (proof of Proposition \ref{prop:cohrig}), and pairwise different. Thus the $\ell$-companions $\sV_s$ are pairwise different as well. This implies that the number of $\sV_s$ is precisely the number of cohomologically rigid connections and finishes the proof.
\end{proof}

\begin{proof}[Proof of Theorem \ref{thm:comp}.]
As in 
Corollary \ref{cor:cohrig} we have a companion assignement  $\Phi(\tau)\colon \sF \mapsto \sV$ which is injective on isomorphism classes (by \v{C}ebotarev density). By Corollary \ref{cor:cohrig} this assignement is a bijection between the set of isomorphism classes of $\sF$ constructed in Theorem \ref{thm:Fstr} and the set of  isomorphism classes of irreducible $\bar \Q_\ell$-lisse sheaves with determinant $L_s$  with the condition 
$H^1(X_{\bar \Q_p},  \Eend^0(\sV_{\bar \Q_p}))=0.$  This shows 1).
We perform the construction  for $\tau'$, yielding $ \Phi(\tau')\colon \sF\mapsto \sV'$. Thus 
$\Phi(\tau')^{-1}\circ \Phi(\tau)( \sF)$ is a $\sigma$-companion to $\sF$.  This shows 2). As 
for 0),      by Theorem \ref{thm:Fstr}, the cardinality of $\sS(r,d)$ is at most the one of $\sS(s,p,r,d)$ while by 
Corollary \ref{cor:cohrig} $\sS(s,\ell,r,d)$ has at most as many elements as $\sS(r,d)$. This shows 0) and thus finishes the proof.
\end{proof}

\section{Concluding observations} \label{sec:rmks}
\subsection{$SL(3)$-rigid connections} \label{ss:LS}

In \cite{EG17} the authors prove Simpson's integrality conjecture for cohomologically rigid flat connections. It was shown in Langer--Simpson's \cite{LS16} that rigid $SL(3)$-connections with integral monodromy are of geometric origin. Combining the two aforementioned results, one sees that cohomologically rigid $SL(3)$-connections on smooth projective varieties are of geometric origin. There is more that can be said in connection to the $p$-curvature conjecture.

\begin{proposition} \label{rmk:YA}
The $p$-curvature conjecture holds for cohomologically rigid connections of rank $3$ and trivial determinant on smooth projective schemes.
\end{proposition}

\begin{proof}
In \cite[Theorem 16.2.1]{And04}, Andr\'e proved that an irreducible  subquotient of a Gau{\ss}-Manin connection  $f\colon Y\to X$ satisfies the $p$-curvature conjecture if $f$ has one complex fibre with connected motivic Galois group.  The result \cite[Theorem 4.1]{LS16} together with the remarks above 
 imply that cohomologically rigid $SL(3)$-connections all are subquotients of Gau{\ss}-Manin connections coming from families of abelian varieties.  Those have a connected motivic Galois group {(see \cite[Proposition 2]{Sch11})}. We conclude that the $p$-curvature conjecture is true for cohomologically rigid $SL(3)$-connections. 
\end{proof}

\subsection{Vanishing of global symmetric forms}
A  smooth projective variety $X$  without global non-trivial $i$-th symmetric differential for all $i$, has the property that all integrable connections are rigid and have finite monodromy, see \cite[Theorem 0.1]{BKT13}.  The proof uses $p$-adic methods to show integrality, and also positivity theorems, ultimately stemming from complex Hodge theory,  as well as $L^2$-methods.  It would be nice to understand at least part of the theorem in terms of characteristic $p$ methods.

\subsection{Motivicity of the isocrystal from good curves to the whole variety}  \label{ss:abe2} The existence of Frobenius structure implies that the isocrystal defined by an irreducible rigid $(E,\nabla)$ on $X_{K_v}$ is motivic on curves  $C^0_s\hookrightarrow X_s$  where $C\hookrightarrow X$ is a dimension $1$  smooth complete intersection  of ample divisors and $C^0_s\hookrightarrow C_s$ is a dense open. See Corollary \ref{cor:mot}. This raises the problem of  extending  the motivicity from $C_s$ to $X_s$. That is, can one find a morphism $f_s\colon Y_s\to U_s$ {over a dense open $U_s\hookrightarrow X_s$},  which has the property  that an irreducible  isocrystal with Frobenius structure and finite determinant on  $X_s$ is a subquotient of the Gau{\ss}-Manin isocrystal $
R^if_{s*}\sO_{Y_s/ \bar \Q_p}$ for some $i\ge 0$? Beyond the study of rigid connections, it would enable one to make progress on the construction of $\ell$-to $p$-companions (see Section \ref{sec:comp}).

\subsection{Rigid connections and the $p$-curvature conjecture}
  

\begin{proposition} \label{prop:pcurv}
Let $X/\mathbb{C}$ be a smooth projective variety.
Then if  any cohomologically rigid connection $(E,\nabla)$ on $X$ has a  model  $(X_S, (E_S, \nabla_S))$  over a scheme $S$ of finite type  over $\Z$ with $p$-curvature $0$ at all closed points $s \in S$,  the monodromy of all cohomologically rigid connections  $(E,\nabla)$ is finite. 
\end{proposition}
\begin{proof}
Let $\rho\colon \pi_1^{\rm top}(X)\to GL(r, \C)$ be a cohomologically rigid representation. By \cite{EG17} the monodromy lies in $GL(r, \sO_L)$ for some number field $L\subset \C$. If $\sigma\colon L\to \C$ is another complex embedding, then the groups $H^1_{\rm an}(X, \Eend^0(\rho)) $ and $ H^1_{\rm an}(X, \Eend^0( \sigma\circ \rho)) $ are equal and thus vanish. Therefore, $\sigma \circ \rho$ is cohomologically rigid. This implies that 
$\bigoplus_{\sigma} \sigma\circ \rho$, where $\sigma$ runs through the Galois group  $G$ of $L$ has monodromy in 
$GL(r|G|, \Z)$ and is unitary by Theorems \ref{thm:unitary}. Thus, applying \cite[Proposition 4.2.1.3]{Kat72}, one concludes that the monodromy of $\bigoplus_{\sigma \in G} \sigma \circ \rho$ is finite. In particular, $\rho$ has finite monodromy. 
\end{proof}

\appendix

\section{Deformations of flat connections in positive characteristic}

{In this appendix we describe an alternative approach to Theorem \ref{thm:nilp} which does not rely on the classical Simpson correspondence. Furthermore, it allows one to deduce stronger statements in the case of cohomological rigidity (see Corollary \ref{cor:p_coh_rigid}).}

Let $Z$ be a smooth projective variety defined over a perfect field $k$ of characteristic $p$.    We denote by $Z'$ is Frobenius twist,  by $F\colon Z\to Z'$ the relative Frobenius.
In the applications, $Z$ is the fibre  $X_s$ at a closed point $s$ of a scheme $S$ of finite type over $\Z$ over which a smooth complex  projective variety $X$ is defined, and $k$ is the residue field of $s$ (which is finite).

\medskip

\begin{remark} \label{rmk:a}
By construction, for $(E,\nabla)$ a rank $r$  integrable connection on $Z$, the point $a\in \A_{Z'}(k)$ from Theorem \ref{thm:BNR} has coordinates $a_i  \in H^0(Z', {\rm Sym}^{i}(\Omega^1_{Z'}))$ from \eqref{1}. 

\end{remark}
  We consider the $\G_m$-action
\ga{5}{ \xymatrix{ \ar[dr] T^*Z' \times_k \G_m \ar[r]^m& T^*Z'\ar[d]^{\pi'}\\
                           & Z'}
}
  defined by  the conic structure on $T^*Z'$.  We  use the notation 
$m\colon V\times_k \G_m\xrightarrow{m} T^*Z'$ for the restriction of $m$ to any subscheme  $V\subset T^*Z'$. For any natural number $n\ge 1$, we define $V_n=V\times_k 
{\Spec}k[t]/(t-1)^n \to V\times_k \G_m$ where $\G_m={\Spec}  k[t,t^{-1}]$. There is a commutative diagram
\ga{6}{ \xymatrix{ \ar[dr] V_n  \ar[r]^m& T^*Z'\ar[d]^{\pi'}\\
                         & Z'.}
}
We also denote by $r\colon V_n\to V$ the retraction obtained via base change from $$\Spec k[t]/(t-1)^n \to \Spec k.$$

\begin{proposition} \label{prop:OV}
Let $Z$ be a smooth projective variety defined over a perfect characteristic $p>0$ field. Let $(E,\nabla)$ be an integrable connection on $Z$ with Hitchin invariant $a=\chi_{dR}((E,\nabla))$. 
If $Z$ lifts to $W_2(k)$ we have an equality of Brauer classes
$$ [m^*\sD_{Z'}] = [r^* \sD_{Z'}] \ \ on  \ \ (Z'_a)_{p}.$$
\end{proposition}
\begin{proof}
By \cite[Proposition 4.4]{OV07}, the Brauer class $[\sD_{Z'}]\in H^2_{\rm \acute{e}t} (T^*Z', \sO^\times_{T^*Z'})$ of the Azumaya algebra $\sD_{Z'}$ is the image $\phi(\theta)$ of the tautological one-form $$\theta \in H^0(T^*Z', \pi^{*}\Omega^1_{Z'}) \subset H^0(T^*{Z'}, \Omega^1_{T^*Z'}) $$  (see Remark \ref{rmk:spec})  by the connecting homomorphism of the \'etale exact sequence
\ga{6}{0\to \sO^\times_{T^*Z'} \to F_*\sO^\times_{T^*Z'}   \xrightarrow{d\log} F_*Z^1(\Omega^1_{T^*Z'}) \xrightarrow{ w^*-C}  \Omega^1_{T^*Z'}\to 0}
on $T^*Z'$, where we use that $T^*Z'$ is the Frobenius twist of $T^*Z$ via $w\colon T^*Z'\to T^*Z$. Recall that $C$ denotes the Cartier operator, $F\colon T^*Z \to T^*Z'$ the relative Frobenius homomorphism, and $Z^1(\Omega^1_{T^*Z'})$ is the sheaf of closed $1$-forms.  We now replace  $T^*Z'$ and $T^*Z$ in \eqref{6} by their product with $\G_m$ over $k$, and the differential forms over $k$ by the ones  over $\G_m$ of the product  varieties.
This yields on $(Z'_a)_p$ 
\ga{}{ m^*\sD_{Z'}=\phi(m^*\theta),  \  r^*\sD_{Z'}=\phi(r^*\theta). \notag}
On the other hand, 
\ga{}{ m^*\theta-r^*\theta \in {\rm Ker}\big( H^0((Z'_a)_p, \Omega^1_{(Z'_a)_p/\Spec k[t]/(t-1)^p} )\to H^0(Z'_a, \Omega^1_{Z'_a})\big).\notag}
Thus, $m^*\theta-r^*\theta $ has support in $Z'_p \hookrightarrow T^*Z\times_k \G_m'$, where $Z'\hookrightarrow T^*Z'$ is the zero-section. We conclude  
\ga{}{ \phi( m^*\theta-r^*\theta ) \in 
H^2_{\rm \acute{e}t}(Z'_p, \sO^\times_{Z'_p}).\notag} 
Hence we have  
\ga{}{\phi( m^*\theta-r^*\theta )=
\phi( m^*\theta|_{Z'_p}-r^*\theta|_{Z'_p} )=m^*\phi(\theta|_{Z'_p})-r^*\phi^*(\theta|_{Z'_p})
 \in 
H^2_{\rm \acute{e}t}(Z'_p, \sO^\times_{Z'_p}).\notag}
As $k$ is perfect, a lifting of $Z$ to $W_2(k)$ is equivalent to a lifting of $Z'$ to $W_2(k)$. 
By \cite[Cor. 2.9]{OV07}, one has the vanishing
\ga{}{ \phi( \theta|_{Z'_p})=0. \notag}
This finishes the proof.
\end{proof}
\begin{remark}
One can show that the choice of a $W_2$-lift in Proposition \ref{prop:OV} above, induces a canonical equivalence of categories 
\begin{equation}\label{ov}\mathsf{QCoh}_{(Z'_a)_p}(m^*\sD_{Z'}) \cong \mathsf{QCoh}_{(Z'_a)_p}(\sD_{Z'}).\end{equation} To see this one applies \cite[Proposition 3.11]{BB07} instead of \cite[Proposition 4.4]{OV07} in the argument above (note that the authors of {\it loc. cit.} call two Azumaya algebras equivalent, if their categories of modules are equivalent).
 Applying \cite[Cor 3.12]{BB07} to the special case of a diagonal morphism, we obtain that for $1$-forms $\theta_1,\theta_2$ on $Z'$ we have a canonical equivalence of categories $\mathsf{QCoh}(\sD_{\theta_1 + \theta_2}) \cong \mathsf{QCoh}(\sD_{\theta_1} \otimes \sD_{\theta_2})$. Putting these two refined assertions together, and evoking the splitting associated by Ogus--Vologodsky to a $W_2$-lift \cite[Cor. 2.9]{OV07} we obtain a canonical equivalence of categories as in \eqref{ov}.
\end{remark}

We denote by $\mu\colon \A_{Z'}\times_k \G_m \to \A_{Z'}$ the action defined for $t\in \G_m$ by multiplication by $t^i$ on $H^0(Z',  {\rm Sym}^i \Omega^1_{Z'})$.
\begin{thm} \label{thm:p_def}
Let $Z$ be a smooth projective variety defined over a perfect characteristic $p>0$ field $k$. Let $(E,\nabla)$ be an integrable connection on $Z$ with Hitchin invariant $a=\chi_{dR}((E,\nabla))$. 
Then 
 assuming that $Z$ lifts to $W_2(k)$,
there exists an integrable connection on $Z\times_k  T$, for $T={\Spec}k[t]/(t-1)^p$, with spectral cover $(Z'_a)_p \xrightarrow{m} T^*Z' \times_k T$, with Hitchin invariant $\mu(a \times_k T )$,  and which restricts to $(E,\nabla)$ on $Z'_a\hookrightarrow T^*Z'$. 
\end{thm}
\begin{proof}
Let $M$ be the {$p_{T^*Z'}\sD_{Z' \times_k T}|_{Z'_a}$-module} associated to $(E,\nabla)$ via the correspondence of Theorem \ref{thm:BNR}.
Then $r^*M$ is a $p^*_{T^*Z'}r^*\sD_{Z'}$-module, thus by Proposition \ref{prop:OV}, $r^*M$ can be viewed as a  $ p^*_{T^*Z'}m^*\sD_{Z'}$-module on $(Z'_a)_p$ (non-canonically). We apply again Theorem \ref{thm:BNR} to conclude to the existence of  an integrable connection on $Z\times_k T$ with the properties of the theorem. This finishes the proof. 
\end{proof}

\begin{remark}\label{rmk:alternative}
Theorem \ref{thm:p_def} gives an alternative proof of Theorem \ref{thm:nilp}  which states that a complex irreducible flat connection $(E_{\mathbb{C}},\nabla_{\mathbb{C}})$, has a model $X_S$ over a finite type scheme $S$ such that the $p$-curvatures at all closed points $s \in S$ are nilpotent. Indeed, by virtue of Theorem \ref{thm:p_def}, every flat connection defined over a characteristic $p$ variety with non-nilpotent $p$-curvature, has a non-trivial deformation of an order which grows linearly with $p$. For $p >> 0$ this exceeds the bound $D$ exhibited Corollary \ref{cor:ndR}.
\end{remark}

We are grateful to one of the anonymous referees for pointing out the following interesting perspective on the material contained in this appendix:

\begin{remark}
Using the methods of \cite{OV07} one can show that the $W_2$-lift of $X_s$ yields a canonical action of $\G_m^{\sharp}$, the PD hull of the neutral element of $\G_m$, on the category $\MIC(X)$. This $\G_m^{\sharp}$-action can be viewed as a de Rham analogue of the $\G_m$-action on the moduli space of Higgs bundles. 
\end{remark}

We remark that the viewpoint on Theorem \ref{thm:nilp} described above still leaves it open whether rigid flat connections have nilpotent $p$-curvature in the case of small primes $p$. For cohomologically rigid flat connections more can be said. The following corollary has been pointed out to us by Y. Brunebarbe.
\begin{corollary}\label{cor:p_coh_rigid}
Let $k$ be a perfect field of positive characteristic $p > 2$ and $Z/k$ a smooth projective $k$-variety with lifts to $W_2(k)$. Let $(E,\nabla)$ be a stable cohomologically rigid flat connection on $Z$. Then $\nabla$ has nilpotent $p$-curvature.
\end{corollary}

An interesting aspect of this corollary is that we only have to assume liftability of $Z$ to $W_2(k)$, and it even applies to small primes $p$. This begs the question whether the same property holds true for all rigid flat connections in positive characteristic, 
without the cohomological assumption, and even without the liftability assumption.

\begin{question}
Let $k$ be a perfect field of positive characteristic $p$ and $Z/k$ a smooth projective $k$-variety. Let $(E,\nabla)$ be an irreducible  stable rigid flat connection of rank $r$ on $Z$. Is the $p$-curvature of  $\nabla$ nilpotent?
\end{question}


\end{document}